\documentclass[11pt]{article}
\usepackage{amstext,amsthm,amsfonts}
\usepackage{amsmath}
\usepackage{geometry}
\usepackage{enumerate}
\usepackage{faktor}
\usepackage[english]{babel}
\usepackage{array}
\usepackage[utf8]{inputenc}
\usepackage{amssymb,etoolbox}
\usepackage{amsopn}
\usepackage{xcolor}
\usepackage{latexsym}
\usepackage{units}
\usepackage{comment}
\usepackage{mathrsfs}
\usepackage{xy}

\usepackage{color}
\usepackage{xfrac}    
\usepackage{nicefrac}

\usepackage{anysize}
\marginsize{2cm}{2cm}{2cm}{2cm}

\usepackage{abstract}

\usepackage{relsize} 

\def\lb{\lambda}

\def\Hi{\mathcal{H}}
\def\de{\mathrm{d}}

\def\AA{\mathfrak A}

\def\R{\mathbb R}
\def\C{\mathbb C}

\def\vect{\operatorname{Vect}}
\def\div{\operatorname{div}}

\newtheorem{Theorem}{Theorem}[section]
\newtheorem{proposition}[Theorem]{Proposition}
\newtheorem{Remark}[Theorem]{Remark}
\newtheorem{Definition}[Theorem]{Definition}
\newtheorem{Lemma}[Theorem]{Lemma}
\newtheorem{Corollary}[Theorem]{Corollary}
\newtheorem*{theorem*}{Theorem} 
\theoremstyle{definition}
\newtheorem{exmp}{Example}[section]

\makeatletter
\newcounter{author}
\renewcommand*\author[1]{%
  \stepcounter{author}%
  \ifnum\c@author=1
    \gdef\@author{#1}%
  \else
    \xdef\@author{\unexpanded\expandafter{\@author\and#1}}%
  \fi
  \csgdef{author@\the\c@author}{#1}}
\newcommand*\email[1]{%
  \csgdef{email@\the\c@author}{#1}}
\newcommand*\address[1]{%
  \csgdef{address@\the\c@author}{#1}}
\AtEndDocument{%
  \xdef\author@count{\the\c@author}%
  \c@author=1
  \print@authors}
\newcommand*\print@authors{%
  \ifnum\c@author>\author@count
  \else
    \print@author{\the\c@author}%
    \advance\c@author by 1
    \expandafter\print@authors
  \fi}
\newcommand*\print@author[1]{%
  \par\medskip
  \begin{tabular}{@{}l@{}}%
    \textsc{Address of \csuse{author@#1}}\\
    \csuse{address@#1}\\
    \textit{E-mail address}:
    \csuse{email@#1}
  \end{tabular}}
\makeatother

\providecommand{\keywords}[1]
{
  {\small	
  \textbf{\textit{Keywords---}} #1}
}

\providecommand{\subjclass}[1]
{
 {\small	
  \textbf{\textit{Mathematics Subject Classification (2020)---}} #1}
}

\begin{document}
\title{Smooth fields of Hilbert spaces, Hilbert bundles and Riemannian Direct Images}

\date{\empty}

\maketitle
\vspace*{-1.3cm}
\begin{minipage}{.93\textwidth}
 \begin{center}
 {\large
     Fabián Belmonte\footnote[1]{\centering Universidad Católica del Norte. \texttt{fbelmonte@ucn.cl}}, 
     \hspace*{2pt} Harold Bustos\footnote[2]{\centering Universidad Austral de Chile. \texttt{harold.bustos@uach.cl}}
     }
 \end{center}
 \vspace{1pt}
\end{minipage}
\begin{center}
\vspace{3pt}
\emph{\large In honour of professor Marius M\u antoiu}  
\end{center}
\vspace{1pt}
\begin{abstract}
\fontsize{10pt}{12pt}\selectfont
Given a field of Hilbert spaces, there are two ways to endow it with a smooth structure: the geometrical notion of Hilbert bundle and the analytical notion of smooth field of Hilbert spaces. We study the relationship between these concepts. We apply our results in the following example: Let $M,N$ be Riemannian oriented manifolds, $\rho:M\to N$ be a submersion and $\pi:E\to M$ a finite-dimensional Hermitian vector bundle with a connection. Also, for each $\lb\in N$, let $M_\lb=\rho^{-1}(\lb)$ and fix a suitable measure $\mu_\lb$ on $M_\lb$. Denote by $\Hi(\lb)$ the Hilbert space of square integrable sections on the restricted Hermitian vector bundle $\pi^{-1}(M_\lb)\to M_\lb$. Does the field of Hilbert spaces $\{\Hi(\lb)\}_{\lb\in N}$ admit a smooth field of Hilbert space structure? or a Hilbert bundle structure?  In order to provide conditions to guarantee a positive answer for these questions, we develop an interesting formula to differentiate functions defined on $N$ as an integral over $M_\lb$.
  
\end{abstract}




\section{Introduction.}

Let $p:H\to N$ be a field of Hilbert spaces, i.e. $p$ is a surjective map such that $\Hi(\lb):=p^{-1}(\lb)$ is a complex Hilbert space. We denote by $\langle\cdot,\cdot\rangle_\lb$ the corresponding inner product on $\Hi(\lb)$, which is antilinear in the second variable. A section is a map $\varphi:N\to H$ such that $(p\circ\varphi)(\lb)=\lb$, for all $\lb\in N$. We denote the set of all sections of such a field by $\Gamma$. For any pair of sections $\varphi,\psi\in\Gamma$, we set the function
$$
h(\varphi,\psi)(\lb)=\langle\varphi(\lb),\psi(\lb)\rangle_\lb.
$$
In order to obtain an interesting mathematical object, we should add further assumptions on a given field of Hilbert spaces. Let us approach this issue first from a geometrical framework recalling the canonical notion of a smooth Hilbert bundle. In what follows, we will also require the notion of Banach manifold, which is defined as the usual notion of manifold but allowing the charts take values in open sets of a Banach space (instead of some Euclidean space). For details, see subsection I.2.1 in \cite{LS} or Chapter III in \cite{SL}.   

\begin{Definition}\label{HB}
Let $p:H\to N$ be a field of Hilbert spaces. Assume that $H,N$ are Banach manifolds and $p$ is a smooth map. Also, let $\{U_i\}_{i\in I}$ be an open cover of $N$, $\{\Hi_i\}_{i\in I}$ be a family of Hilbert spaces, and for each $\lb\in U_i$, $T_i(\lb):\Hi(\lb)\mapsto \Hi_i$ be a unitary operator. The family $(U_i,\Hi_i,T_i)_{i\in I}$ is smooth Hilbert bundle structure on $p:H\to N$ if the map $\mathcal{T}_i:p^{-1}(U_i)\to U_i\times \Hi_i$ given by $\mathcal{T}_i(x)=\left(p(x),T_i(p(x))(x)\right)$ is a diffeomorphism, for every $i\in I$.

A (smooth) Hilbert bundle is a field of Hilbert spaces endowed with a Hilbert bundle structure. If $p:H\to N$ is a Hilbert bundle,  we denote by $\Gamma^\infty(N,H)$ the space of smooth sections $\varphi:N\to H$. 
\end{Definition}

We set the field of unitary operators $T_i(\lb)$ instead of the local trivialization maps $\mathcal T_i$ as customary in the literature, because this will become useful later. 

In the rest of this article we will assume that the manifold $N$ has finite dimension. We will denote by $C^\infty(N)$ the space of smooth complex valued functions on $N$ and by  $\operatorname{Vect}(N)$ the complexified space of smooth vector fields on $N$.

We obtain a more interesting and subtle concept when we also assume the existence of a connection over a Hilbert bundle.  

\begin{Definition}\label{HeB}
A Hilbert bundle with a connection is a Hilbert bundle $p:H\to N$ together with a map
$\nabla:\operatorname{Vect}(N)\times\Gamma^\infty(N,H)\to\Gamma^\infty(N,H)$ such that, for $X,Y\in\operatorname{Vect}(N)$, $a\in C^\infty(N)$ and $\varphi,\psi\in\Gamma^\infty(N,H)$:
\begin{enumerate}
\item[i)] $\nabla_{X+Y}=\nabla_X +\nabla_Y$, $\nabla_{a X}=a\nabla_X$, $\nabla_X(a\varphi)=X(a)\varphi+a\nabla_X(\varphi)$
\item[ii)] $h(\varphi,\psi)\in C^\infty(N)$ and $Xh(\varphi,\psi)=h(\nabla_X\varphi,\psi)+h(\varphi,\nabla_{\overline{X}}\psi)$

\end{enumerate}

\end{Definition}


Let $p:H\to N$ be a field of Hilbert spaces. If all Hilbert spaces $\Hi(\lb)$ admit an orthonormal basis with the same cardinality $J$, then one can endow $p:H\to N$ with an artificial Hilbert bundle structure. Indeed, fix an orthonormal basis $\{\varphi_j(\lb)\}_{j\in J}$ of $\Hi(\lb)$, for each $\lb\in N$. Then, computing the corresponding Fourier coefficients on each $\Hi(\lb)$ defines a unitary operator $T(\lb):\Hi(\lb)\to l^2(J)$. We can endow $H$ with a Banach manifold structure by taking the map $x\to (p(x),T(p(x))x)$ as a global chart; then, by definition, $p:H\to N$ becomes a Hilbert bundle.    



We aim to construct Hilbert bundles structures more intrinsically. However, let us first approach the problem of introducing a notion of smoothness for a field of Hilbert spaces in a more analytical manner. Consider the trivial bundle: $H=N\times \Hi$. In this case, a section is a function $\varphi:N\to\Hi$ and the space of smooth sections is $C^\infty(N,\Hi)$.
Heuristically, in this scenario, we know how to differentiate certain sections before endowing the space $H$ with a smooth structure. Generally, given a field of Hilbert spaces, we might require the existence of a space of sections where we have an apriori way to compute a sort of derivation. The approach of endowing the field of Hilbert spaces with a suitable space of sections was applied by Dixmier and Douady long ago \cite{D3} to solve the analogue problem in the topological framework, using the notion of continuous field of Hilbert spaces defined by Godement \cite{God}.  Surprisingly, the problem in the smooth category was not considered until a few years ago \cite{LS}, where L. Lempert and R. Sz\H{o}ke introduced the following notion of smooth field of Hilbert spaces.
\begin{Definition}\label{def}
Let $N$ be a finite-dimensional smooth manifold. A smooth structure on a field of Hilbert spaces $p:H\to N$ is given by specifying a set of sections $\Gamma^\infty$, closed under addition and under multiplication by elements of $C^\infty(N)$, and a map $\nabla:\operatorname{Vect}(N)\times\Gamma^\infty\to\Gamma^\infty$ such that, for $X,Y\in\operatorname{Vect}(N)$, $a\in C^\infty(N)$ and $\varphi,\psi\in\Gamma^\infty$:
\begin{enumerate}
\item[i)] $\nabla_{X+Y}=\nabla_X +\nabla_Y$, $\nabla_{a X}=a\nabla_X$, $\nabla_X(a\varphi)=X(a)\varphi+a\nabla_X(\varphi)$
\item[ii)] $h(\varphi,\psi)\in C^\infty(N)$ and $Xh(\varphi,\psi)=h(\nabla_X\varphi,\psi)+h(\varphi,\nabla_{\overline{X}}\psi)$
\item[iii)] $\Hi^\infty(\lb):=\{\varphi(\lb)\mid \varphi\in\Gamma^\infty\}$ is dense in $\Hi(\lb)$, for all $\lb\in N$.
\end{enumerate} 
We call a triple $(p:H\to N,\Gamma^\infty,\nabla)$ having these properties a smooth field of Hilbert spaces with connection $\nabla$.
\end{Definition}

Our first goal is as follows: given a smooth field of Hilbert spaces $(p:H\to N,\Gamma^\infty,\nabla)$, to find conditions as sharp as we can to guarantee the existence of a Hilbert bundle structure on $p:H\to N$ with connection $\tilde \nabla$ such that $\Gamma^\infty\subseteq\Gamma^\infty(N,H)$ and $\tilde\nabla\mid_{\Gamma^\infty}=\nabla$. L. Lempert and R. Sz\H{o}ke approached this problem, but assuming that the required trivialization is globally defined (i.e., defining $T(\lb)$ for every $\lb\in N$, instead of finding an open covering $\{U_i\}_{i\in I}$ and defining $T_i(\lb)$ for every $i\in I$ and $\lb\in U_i$, see Definition \ref{HB}) and that the connection is flat or projectively flat \cite[Th. 2.4.2]{LS}. We will approach that problem in Section \ref{section2}, without assuming these conditions. Supposing that the trivialization is only locally defined will lead to technical problems that were not considered in \cite{LS}, which we will address by studying the so-called transition maps.

In general, the idea of describing a geometrical object in terms of an analytic-algebraic object is important in many areas of modern mathematics, especially in Noncommutative Geometry \cite{GVF,Kh}. For instance, the Serre-Swan Theorem asserts that the category of (finite-dimensional) complex vector bundles on a compact Hausdorff space $X$ is equivalent to the category of finitely generated projective modules over $C(X)$ \cite{Sw, GVF, Kh}. The required functor maps a vector bundle into its corresponding space of (continuous) sections.
In a certain sense, our purpose is to describe vector bundles in terms of a test sections space $\Gamma^\infty$ forming a $C^\infty(N)$-module, much as in the Serre-Swan Theorem, but in the smooth framework, allowing Hilbert spaces as fibers, and requiring a connection defined on $\Gamma^\infty$. 

One of the advantages of the notion of a smooth field of Hilbert spaces is that it allows the construction, for each $n\in\mathbb N_0\cup\{\infty\}$, of the Fr\'echet spaces of $n$-times differential sections $\Gamma^n(N)$. Essentially,  $\Gamma^n(N)$ is obtained adding to $\Gamma^\infty$ all the sections that are locally uniform limit of some sequence $(\varphi_k)$ in $\Gamma^\infty$ such that $\nabla_{X_1}\nabla_{X_2}\cdots \nabla_{X_j}(\varphi_k)$ also converges locally uniformly, for any $j\leq n$ and $X_1,X_2,\cdots, X_j\in \vect{(N)}$
(see the beginning of subsection \ref{bsfHO} for details). Another basic property is that, for any open subset $U\subseteq N$, the restriction $\Gamma^\infty|_U$ with the connection given by $\nabla^{U}_X\left(\varphi|_U\right):=(\nabla_X\varphi)|_U$ defines a smooth structure on the restricted field of Hilbert spaces $p:H|_U\to U$; we will denote by $\Gamma^n(U)$ the corresponding space of $n$-times differential sections. Eventually, we will prove the following result. 

\begin{Theorem}\label{porpeqivcor}
Let $(p:H\to N,\Gamma^\infty,\nabla)$ be a smooth field of Hilbert spaces. Also, let $\{U_i\}_{i\in I}$ be an open cover of $N$ and $\{\Hi_i\}_{i\in I}$ be a family of Hilbert spaces, and assume that for each $\lb\in U_i$, there is a unitary operator $T_i(\lb):\Hi(\lb)\to\Hi_i$. The family $\{(U_i,\Hi_i,T_i)\}_{i\in I}$ defines a Hilbert bundle structure with connection $\tilde\nabla$ such that $\Gamma^\infty(U_i,H|_{U_i})=\Gamma^\infty(U_i)$ and $\tilde\nabla\mid_{\Gamma^\infty}=\nabla$ if and only if $T_i(\Gamma^\infty(U_i))=C^\infty(U_i,\Hi_i)$ for every $i\in I$. If any of these conditions hold, $\Gamma^\infty(N)=\Gamma^\infty(N,H)$. 

\end{Theorem}


Most of the techniques that we developed to approach our problems, including finding conditions to apply the previous Theorem, will require to extend the notion of smooth field of operators introduced in \cite{BBC}, which we briefly recall here. Let $(p_1:H_1\to N,\Gamma^\infty_1,\nabla^1)$ and $(p_2:H_2\to N,\Gamma^\infty_2,\nabla^2)$ be smooth fields of Hilbert spaces, and also let $A=\{A(\lb)\}_{\lb\in N}$ be a field of operators such that $\Hi_1^\infty(\lb)\subseteq D[A(\lb)]$ and $\Hi_2^\infty(\lb)\subseteq D[A^*(\lb)]$, where $D[A(\lb)]$ and $D[A^*(\lb)]$ are the domains of the operators $A(\lb)$ and $A^*(\lb)$ respectively, for every $\lb\in N$. We say that $A$ is a smooth field of operators if $A(\Gamma_1^\infty)\subseteq \Gamma_2^\infty(N)$  and $A^*(\Gamma_2^\infty)\subseteq \Gamma_1^\infty(N)$, where $(A\varphi)(\lb):=A(\lb)\varphi(\lb)$ and $(A^*\psi)(\lb)=A^*(\lb)\psi(\lb)$ for every $\varphi\in\Gamma_1^\infty$ and $\psi\in\Gamma_2^\infty$. We summarize details and some of the main general features of smooth fields of operators in subsection \ref{bsfHO}. For instance (Proposition \ref{Con}), it turns out that the connections on the underlying fields of Hilbert spaces induce a well-defined connection $\hat\nabla$ on the space of smooth fields of operators, given by
$$
\hat\nabla_X(A)=\nabla_X^2 A-A\nabla_X^1\,,\qquad X\in\vect(N).
$$
In Section \ref{section2}, we approach our problem in two steps: first, looking for conditions to guarantee the existence of a suitable Hilbert bundle structure, and second, looking for (stronger) conditions to guarantee the existence of a suitable connection (subsection \ref{Hb}).
Let $N$ be a manifold and $p:H\to N$ be a field of Hilbert spaces. Assume that there are given an open cover $\{U_i\}_{i\in I}$ of $N$, a family of Hilbert spaces  $\{\Hi_i\}_{i\in I}$, and for each $\lb\in U_i$, a unitary operator $T_i(\lb):\Hi(\lb)\to\Hi_i$. The existence of a suitable Hilbert bundle structure is usually reformulated in terms of the so called transition maps $\tau_{i,j}(\lb):=T_j(\lb)T^*_i(\lb)$, $\lb\in U_i\cap U_j$. We will show in Proposition \ref{bundle} that, in order to obtain a Hilbert bundle structure, it is enough to require that the map $\lb\mapsto \tau_{ij}(\lb)$ be strongly smooth, i.e the map $\lb\mapsto \tau_{ij}(\lb)x$ belongs to $C^\infty(U_i\cap U_j,\Hi_j)$, for every $x\in\Hi_i$.

We will find a condition equivalent to strong smoothness assuming that each $\tau_{ij}$ defines a smooth field of operators on a dense smooth submodule of $C^\infty(U_i\cap U_j,\Hi_i)$. More precisely, when the fiber spaces are constants (i.e. $\Hi_1(\lb)=\Hi_1$ and $\Hi_2(\lb)=\Hi_2$),  the connection $\hat{\nabla}$  will allow us to characterize strong smoothness in terms of the consecutive derivatives of the field of operators under consideration (see Proposition \ref{Der}), and we apply that characterization in Corollary \ref{Xtau} over the field of operators $\tau_{ij}$. 

Instead of dealing with the field of operators $\tau_{ij}$, it seems more practical to directly analyze the field of operators $T_i$. In Proposition \ref{SF-HB}, applying our smooth field of operators notion and $\hat\nabla$, we will provide natural conditions over each $T_i$ to ensure the existence of a Hilbert bundle structure on the underlying field of Hilbert spaces. 

In order to provide conditions to guarantee the existence of a Hilbert bundle structure with a connection, under suitable assumptions, it is useful to introduce the fields of operators $\tilde A(X):=-\hat\nabla_X(T_i)T_i^*$, for every $X\in\vect{(U_i)}$. Essentially, we will prove that the family $(U_i,\Hi_i,T_i)$ allows to construct a Hilbert bundle structure with a suitable connection if and only if each $\tilde A_i(X)$ maps $C^\infty(U_i,\Hi_i)$ into itself. As a consequence, we will obtain Theorem \ref{porpeqivcor} and Corollary \ref{corisoc}, which will be applied later to show that the examples given in Section \ref{RDI} define Hilbert bundles with a connection.   

One of the main outcomes of this article is the introduction of a proper definition of trivialization for a smooth field of Hilbert spaces. Actually, we will provide two different notions: the notion of weak smooth local trivialization is meant to imply the existence of a Hilbert bundle structure (see Definition \ref{wLT}), and it comes from Proposition \ref{SF-HB}; the notion of smooth local trivialization follows from requiring the conditions in Theorem \ref{porpeqivcor}, and it is obviously meant to imply the existence of Hilbert bundle structure with a connection (see Definition \ref{LT}). A notion of trivialization was also introduced in \cite{LS}, but only considering the global and (projectively) flat case. 

At the end of Section \ref{section2}, we will provide an abstract result concerning the construction of a smooth trivialization (see Theorem \ref{HS}). Once again, this problem was approached in \cite{LS}, but in the global and flat case. Nevertheless, from there we borrow the idea of considering the space of horizontal sections, but in a general setting. However, we have not yet been able to establish conditions to guarantee the existence of a horizontal section passing through every point in this generality (flatness is fundamental in the proof of that result in \cite{LS}).


Our second goal is to provide a large family of examples, where we can apply our results to construct smooth fields of Hilbert spaces, smooth trivializations and therefore Hilbert bundles with a connection. In Section \ref{RDI} we develop such examples, which we call Riemannian direct images. We borrow part of the name from \cite{LS}, where an analogous example was developed in the holomorphic framework. Let $M,\,N$ be oriented Riemannian manifolds and $\rho:M\to N$ be a surjective smooth submersion (but $\rho$ is not necessarily a Riemannian submersion). Assume there is a finite-dimensional Hermitian vector bundle $\pi:E\to M$ with a given Hermitian connection $\nabla^{E}$. We set $M_\lb=\rho^{-1}(\lb)$ and a volume form $\mu_\lb$ on $M_\lb$. For each $\lb\in N$, let $ E_\lb:= E|_{M_\lb}$ and consider the Hilbert space $\Hi(\lb):=L^2(E_\lb,\mu_\lb)$ consisting of the square integrable sections on $M_\lb$. In order to construct a smooth structure on the field of Hilbert spaces $\Hi(\lb)$, it is necessary to construct a space of sections $\Gamma^\infty$ and to compute the derivative $Xh(\varphi,\psi)$, for $X\in\vect(N)$ and $\varphi,\psi\in\Gamma^\infty$. Since the inner products on each $\Hi(\lb)$ are an integral over $M_\lb$ and we are looking for a connection satisfying condition \emph{ii)} in Definition \ref{def}, we would like to compute derivatives of functions  $F:N\to\R$ of the form 
\begin{equation}\label{eqF}
F(\lb):=\int_{M_\lb}\,f\,d\mu_{\lb},
\end{equation}
for every $f\in C_c^\infty(M)$. Since $\rho$ is a submersion, $D\rho(x)$  is an isomorphism between the spaces $[T_x M_\lb]^{\perp}$ and $T_{\lb}N$, where $\lb=\rho(x)$ and the orthogonal complement is computed using the Riemannian structure on $M$. This implies that we can lift a vector field $X$ on $N$ to a vector field $\check X$ on $M$, pointing in the normal direction at each $M_\lb$. Let $J(x)$ denote the determinant of the restriction of $D\rho(x)$ to the spaces $[T_x M_{\rho(x)}]^{\perp}$. Using the divergence Theorem and the coarea formula (see Appendix \ref{appcoarea}), we will show the following result, which we find interesting in its own right:

\begin{Theorem}\label{smoothint}
Let $M,\,N$ be oriented Riemannian manifolds, $\rho\in C^m(M,N)$ be a submersion and $f\in C_c^r(M)$, with $1\leq r\leq m$. The map $F:N\to\R$ defined by \eqref{eqF} is $r$ times differentiable, where $\mu_\lb=J^{-1}\eta_\lb$ and $\eta_\lb$ is the volume form on $M_\lb$ induced by its Riemannian structure. Moreover, if $X$ is a vector field on $N$, then
$$
X F(\lb)=\int_{M_\lb}\check X(f)+(\div(\check X)-\div X(\lb))f\,\mu_{\lb},
$$
where the divergence $\div(\check X)$ is computed with respect to the Riemannian volume form on $M$, while the divergence $\div(X)$ is computed with respect to the Riemannian volume form on $N$.

\end{Theorem}
It is worth to note that the term $\div(\check X)-\div X\circ\rho$ in a certain sense coincides with the divergence of $\check X$ with respect to any $(m-k)$-form $\nu$ on $M$ such that $\eta=\nu\wedge\rho^*\zeta$, where $\eta$ is the Riemannian volume form on $M$ and $\zeta$ is the Riemannian volume form on $N$ (for details see Lemma \ref{l2} and the paragraphs before it). We denote by $\div_\nu(\check X)$ that divergence. Thus, our derivation formula implies that $X F(\lb)=\int_{M_\lb}\div_\nu(f \check X)\,\mu_\lb$.

As we mentioned before, we will apply the derivation formula in Theorem \ref{smoothint} to prove that we can define an explicit smooth structure on the field of Hilbert spaces $\Hi(\lb)=L^2(E_\lb)$. In order to do so, we will take as space of sections $\Gamma^\infty=\Gamma_c^\infty(E)$, i.e. the space of smooth compact supported sections of $E$ (we identify $\varphi\in \Gamma_c^\infty(E)$  with the section given by $\varphi(\lb)=\varphi|_{M_\lb}$). More precisely, in subsection \ref{secsst} we will prove the following result.
\begin{Theorem}\label{hilejem}
Let $\pi:E\to M$ be a finite-dimensional Hermitian bundle with a Hermitian connection $\nabla^{E}$\, and let $\rho:M\to N$ be a submersion, where $M,\,N$ are oriented Riemannian manifolds. Also, let $\Gamma^{\infty}=\Gamma_c^{\infty}(E)$ and $\Hi(\lb)=L^2(E_\lb)$, where $M_\lb$ is endowed with the volume form $\mu_\lb$ defined in \ref{xjmu}. The map $\nabla :\vect(N)\times\Gamma^{\infty}\to \Gamma^{\infty}$ given by
$$
\nabla_{X}( \varphi)=\nabla^{E}_{\check X}( \varphi)+\frac{1}{2}(\div(\check X)-\div(X)\circ\rho) \varphi,
$$
defines a smooth field of Hilbert spaces structure on $H\to N$ with curvature $R[X,Y]=R^E[\check{X},\check{Y}]$, where $R^E$ denotes the curvature of $\nabla^{E}$. 
\end{Theorem}

In order to construct a Hilbert bundle structure on the field of Hilbert spaces $\Hi(\lb)=L^2(E_\lb)$, we assume that $\rho$ is a fiber bundle. In that case, if $F$ is the fiber of $\pi$ and $K$ is the fiber of $\rho$, then we can explicitly define unitary operators $T_i(\lb):L^2(E_\lb)\to L^2(K,F)$ for every $\lb\in U_i$, where $\{U_i\}_{i\in I}$ is a suitable open cover of $N$. We will give conditions to guarantee that $\{(U_i,L^2(K,F),T_i)\}_{i\in I}$ defines a smooth trivialization for our smooth field of Hilbert spaces in Theorem \ref{trivex1}. In particular, we have the following direct consequence.     

\begin{proposition}\label{prosub}
Let $\pi:E\to M$ be a finite-dimensional Hermitian bundle with a Hermitian connection $\nabla^{E}$, and let $\rho:M\to N$ be a proper submersion, where $M,\,N$ are oriented Riemannian manifolds. Then the smooth field of Hilbert spaces $\Hi(\lb)=L^2(E_\lb)$ admits a local smooth trivialization. In particular, $H\to N$ is a Hilbert bundle with a connection.   
\end{proposition}

We finish subsection \ref{secsst} discussing two particular cases: the trivial line bundle $E=M\times F$, with $F$ any finite-dimensional vector space; and the tangent bundle $E=TM$. In the first case, since the sections of $E\to M$ are functions $\varphi:M\to F$, we have that $\Hi(\lb)=L^2(M_\lb,F)$ and $\Gamma^\infty=C_c^\infty(M,F)$. It will become clear that if $\rho$ is a fiber bundle (not necessarily proper) then $H\to N$ is a smooth field of Hilbert spaces admitting a (flat) smooth local trivialization; in particular, $H\to N$ is a Hilbert bundle with a connection (see Corollary \ref{TlB}). In the second particular case, we use the Levi-Civita connection $\nabla^L$ on the tangent bundle $TM$, and we rewrite the conditions of Theorem \ref{trivex1} in terms of the Christoffel symbols of $\nabla^L$ (see Corollary \ref{LeCi}).


\subsection{Basic properties of smooth fields of Hilbert spaces and smooth fields of operators}\label{bsfHO}
Unlike in the introduction, in what follows we will usually denote a field of Hilbert spaces by $H\to N$, omitting the letter $p$.  

Let $(H\to N,\Gamma^\infty,\nabla)$ be a smooth field of Hilbert spaces and let $U\subseteq N$ be an open set. Let us recall the definition of the space $\Gamma^n(U)$ given in subsection 3.1 in \cite{LS}, $n\in\mathbb N\cup\{\infty\}$. The space $\Gamma^0(U)$ is the $C(U)$-module of those sections  of $H$ that are locally uniform limits of a sequence in $\Gamma^\infty$. The space $\Gamma^1(U)$ is the $C^1(U)$-module of those $\varphi\in \Gamma^0(U)$ for which there is a sequence $\varphi_j\in\Gamma^\infty$ such that $\varphi_j\to \varphi$ locally uniformly, and for every $X\in \text{Vec}(U)$, the sequence $\nabla_X\varphi_j$ converges locally uniformly. For such $\varphi$, we can define $\nabla_X\varphi=\lim\nabla_X\varphi_j$ (Lemma 3.1.2 in \cite{LS}). The space $\Gamma^n(U)$ is defined inductively: $\varphi\in \Gamma^n(U)$ if $\varphi$ and $\nabla_X\varphi$ belongs to $\Gamma^{n-1}(U)$, for all $X\in\text{Vect}(U)$. Finally,
$\Gamma^\infty(U)=\bigcap\Gamma^n(U)$.
The spaces $\Gamma^n(U)$ and $\Gamma^\infty(U)$ are Fréchet spaces with the seminorms defined by
\begin{equation}\label{SN}
||\varphi||_{C,X_1,\cdots X_m}=\sup\{||\nabla_{X_1}\cdots\nabla_{X_m}\varphi(\lb)||:\lb\in C\},   
\end{equation}

\noindent where $C\subseteq U$ is compact, $X_1,\cdots,X_m\in\operatorname{Vect}(U)$ and $m\leq n$ (we can take $X_1,\cdots X_m\in\Xi$, with $\Xi\subset \text{Vect}(U)$ finite and generating the tangent space at each $\lb\in U$). 
\begin{Remark}
{\rm
Lemma 2.2.3 in \cite{LS} implies that the restriction $\Gamma^\infty|_U$, together with the connection given by $\nabla^{U}_X(\varphi|_U)=\nabla_X(\varphi)|_U$, defines a smooth structure on $H|_U\mapsto U$.
}
\end{Remark}
It will become useful to extend the notion of smooth field of operators given \cite{BBC}. Recall that if $H^1\to N$ and $H^2\to N$ are fields of Hilbert spaces and $A=\{A(\lb)\}$ is a field of operators (i.e. $A(\lb)$ is an operator with domain in $H^1(\lb)$ and range in $H^2(\lb)$), then $A$ can be interpreted as a map sending suitable sections into sections, defining $(A\varphi)(\lb):=A(\lb)\varphi(\lb)$. 

\begin{Definition}\label{SFO}
Let $(H^1\to N,\Gamma^\infty_1,\nabla^1)$ and $(H^2\to N,\Gamma^\infty_2,\nabla^2)$ be smooth field of Hilbert spaces. Also let  
 $A=\{A(\lb)\}$ be a field of operators with domains $D[A(\lb)]\subseteq \Hi_1(\lb)$ and $A(\lb):D[A(\lb)]\to\Hi_2(\lb)$, for each $\lb\in N$. We say that $A$ is $n$-times smooth if  
\begin{enumerate}
 \item[i)] $D[A(\lb)]$ contains $\Hi_1^\infty(\lb):=\{\varphi(\lb)\mid \varphi\in \Gamma_{1}^\infty\}$ and $D[A^*(\lb)]$ contains $\Hi_{2}^\infty(\lb)$.
\item[ii)] $A(\Gamma_{1}^\infty)\subseteq \Gamma_2^n(N)$ and $A^*(\Gamma_{2}^\infty)\subseteq \Gamma_1^n(N)$.
\end{enumerate}
We denote by $\AA^n(\Gamma^\infty_1,\Gamma^\infty_2)$ the space of $n$-times smooth fields of operators. We say that a field of operators $A$ is smooth if A belongs to $\AA^n(\Gamma^\infty_1,\Gamma^\infty_2)$, for every $n\in\mathbb N$. We denote by $\AA^\infty(\Gamma^\infty_1,\Gamma^\infty_2)$ the space of  smooth fields of operators. We also define $\AA^n(\Gamma^\infty_1)=\AA^n(\Gamma^\infty_1,\Gamma^\infty_1)$ and $\AA^\infty(\Gamma_1^\infty)=\AA^\infty(\Gamma^\infty_1,\Gamma^\infty_1)$. 
\end{Definition}

\begin{Remark}\label{cmoduleop}
{\rm
Let $A$ be an $n$-times smooth field of operators. For a given function  $a\in C^{\infty}(N)$, by using the rule $\nabla_X(aA\varphi)=X(a)A\varphi+a\nabla_X(A\varphi)$, one can show that the operator $aA$ is also a $n$-smooth field of operators. Moreover, the $n$-times fields of operators form a $C^{\infty}(N)$-module. See \cite{BBC} for more properties  
}
\end{Remark}

Let $A$ be a smooth field of operators. Then, for each $X\in\vect{(N)}$, the operator $\hat\nabla_X(A)=\nabla^2_X A-A\nabla^1_X$ also maps $\Gamma_{1}^\infty$ into $\Gamma_{2}^\infty(N)$, and by definition we have the Leibniz's rule
\begin{equation}\label{Leib}
\nabla^2_X(A\varphi)=\hat\nabla_X(A)\varphi+A\nabla^1_X(\varphi),   
\end{equation}
for every $\varphi\in\Gamma_{1}^\infty$. The following result summarizes the main properties of $\hat\nabla$ and  is a straightforward generalization of Theorem 2.6 in \cite{BBC}. 

\begin{proposition}\label{Con}
Let $(H^1\to N,\Gamma^\infty_1,\nabla^1)$ and $(H^2\to N,\Gamma^\infty_2,\nabla^2)$ be smooth field of Hilbert spaces and $A\in \AA^\infty(\Gamma^\infty_1,\Gamma^\infty_2)$. Then $\hat\nabla_X(A)=[\nabla_X,A]$ is also a smooth field of operators and has the following properties for all $X,Y\in \operatorname{Vect}(N)$, $f\in C^{\infty}(N)$.
\begin{enumerate}
\item[i)] $\hat\nabla_{X+Y}(A)=\hat\nabla_{X}(A)+\hat\nabla_{Y}(A)\,, \quad ii)\;\hat\nabla_{fX}(A)=f\hat\nabla_{X}(A)$ 
\item[ii)] $\hat\nabla_{X}(fA)=X(f)A+f\hat\nabla_{X}(A),$
\item[iii)] $h(\hat\nabla_{X}(A)\varphi,\psi)=h(\varphi,\hat\nabla_{\overline{X}}(A^{*})\psi)$, for every $\varphi\in\Gamma_1^\infty$ and $\psi\in\Gamma_2^\infty$.
\item[iv)]$\hat\nabla_{\overline{X}}(A^{*})(\lb)\subseteq[\hat\nabla_{X}(A)(\lb)]^*$, for each $\lb\in N$.
\end{enumerate}
\end{proposition}
The proof of the previous result is practically the same as 
the proof of Theorem 2.6 in \cite{BBC}. Nevertheless, let us point out that {\it i)} and {\it ii)} follow from a direct computation, and {\it iii)} follows from condition {\it ii)} in Definition \ref{def}. In order to show that $\hat\nabla_X(A)$ is a field of operators, for each $\lb\in N$ and $u\in \Hi_1^\infty(\lb)$, we define 
$$
\hat\nabla_X\left(A\right)(\lb)u=(\hat\nabla_X(A)\varphi)(\lb),
$$
where $\varphi\in\Gamma^\infty$ is such that $\varphi(\lb)=u$. Notice that property {\it iii)} and the density of $\Hi^\infty_2(\lb)$ in $\Hi_2(\lb)$ implies that if $\varphi(\lb)=0$, then $\hat\nabla_X(A)\varphi(\lb)=0$. The latter implies that $\hat\nabla_X(A)(\lb)$ is well defined (independent of the required $\varphi)$. 

.

\begin{Remark}\label{sder}
{\rm
Let us explain the concepts we have introduced so far by considering the trivial case, i.e., $H=N\times \Hi$, where $\Hi$ is a Hilbert space. As we mentioned before, in this case, the space of smooth sections is $C^\infty(N,\Hi)$, and $\nabla_X=X$ makes $H$ a Hilbert bundle with a connection. In what follows we say that $\Gamma^\infty$ is a smooth $C^\infty(N)$-submodule of $C^\infty(N,\Hi)$ if $\Gamma^\infty$ is a subspace in $C^\infty(N,\Hi)$ such that $f\varphi$ and $X(\varphi)$ belong to $\Gamma^\infty$, for every $f\in C^\infty(N)$ and  $X\in\vect{(N)}$. 
\begin{enumerate}
    \item[a)]  If $\Gamma^\infty$ is a dense $C^\infty(N)$-submodule of $C^\infty(N,\Hi)$, then $\Gamma^\infty$ defines a smooth field of Hilbert spaces structure on $H=N\times\Hi$. Indeed, condition {\it iii)} in Definition \ref{def} follows by considering vectors in $\Hi$ as constant sections. Moreover, by definition $\Gamma^\infty(N)= C^\infty(N,\Hi)$. We expect that the converse claim holds true as well (i.e. every space of sections defining a smooth field of Hilbert spaces on $N\times \Hi$ with $\Gamma^\infty(N)= C^\infty(N,\Hi)$ is a dense $C^\infty(N)$-submodule in $C^\infty(N,\Hi)$), but we will not need such a result, at least in this article.

    \item[b)] If $\Gamma_1^\infty$ is a dense smooth $C^\infty(N)$-submodule of $C^\infty(N,\Hi_1)$, $\Gamma_2^\infty$ is a dense smooth $C^\infty(N)$-submodule of $C^\infty(N,\Hi_2)$ and $\Gamma_1^\infty$ contains the constant sections, then $\hat\nabla$ coincide with the canonical strong derivation. Indeed, if $A\in\AA^1(\Gamma^\infty_1,\Gamma^\infty_2)$ then equation \eqref{Leib} implies that 
$$
X(Au)=\hat\nabla_X(A)u,
$$
for any $u\in\Hi_1$. In what follows, in the trivial case we will denote by $\hat X(A)$ the field of operators $\hat\nabla_X(A)$.
\end{enumerate}

}    
\end{Remark}

\begin{exmp}\label{exmp}
Let $D\subseteq \Hi$ be a dense subspace. An important example of a dense smooth $C^\infty(N)$-submodule is the space $P^\infty(N,D)$ of all the sections $f\in C^\infty(N,\Hi)$ of the form $f(\lb)=\sum_j^n a_k(\lb)x_k$, with $a_k\in C^\infty(N)$ and $x_k\in D$. Indeed, the case $D=\Hi$ follows from \cite[Proposition 44.2]{Tre} and the general case follows from an $\epsilon/2$-argument.
\end{exmp}
The following trivial result asserts that Leibniz's rule also holds for $\hat\nabla$ whenever it makes sense.

\begin{proposition}\label{Leib1}
Let $(H^1\to N,\Gamma^\infty_1,\nabla^1)$, $(H^2\to N,\Gamma^\infty_2,\nabla^2)$ and $(H^3\to N,\Gamma^\infty_3,\nabla^3)$ be smooth field of Hilbert spaces. Also, let $B\in\AA^\infty(\Gamma^\infty_1,\Gamma^\infty_2)$. If $A:\Gamma^\infty_2(N)\to\Gamma^\infty_3(N)$, then 
$$
\hat\nabla_X(AB)=\hat\nabla_X(A)B+A\hat\nabla_X(B),
$$
for every $X\in\vect{(N)}$. The same identity holds if $B:\Gamma^\infty_1\to\Gamma^\infty_2$ and $A\in\AA^\infty(\Gamma^\infty_2,\Gamma^\infty_3)$.
\end{proposition}

The following result provides some general conditions on a smooth field of operators in terms of their derivatives to guarantee that it can be extended to the completed space $\Gamma_1^\infty(N)$. In the next subsection we will find more subtle conditions when $H^2\to N$ is trivial. But first, we require the following definition. 
\begin{Definition}
A field of bounded operators $A=\{A(\lb)\}_{\lb\in N}$ is called uniformly bounded if there is a common bound for the norms of those operators. A field of operators $A=\{A(\lb)\}_{\lb\in N}$ is called locally uniformly bounded if for every $\lb\in N$ there is an open set $U\subseteq N$ such that $A|_U$ is uniformly bounded.    
\end{Definition}
\begin{proposition}\label{lub}
Let $(H^1\to N,\Gamma^\infty_1,\nabla^1)$ and $(H^2\to N,\Gamma^\infty_2,\nabla^2)$ be smooth field of Hilbert spaces and fix $n\in\mathbb{N}\cup\{0\}$. If $A\in \AA^n(\Gamma^\infty_1,\Gamma^\infty_2)$ and $\hat\nabla_{X_1}\cdots \hat\nabla_{X_k}(A)$ is a locally uniformly bounded fields of operators, for every $k\leq n$ and $X_1,\cdots, X_k\in\vect{(N)}$, then $A\left(\Gamma_1^n(N)\right)\subseteq \Gamma_2^n(N)$ and $A:\Gamma_1^n(N)\to \Gamma_2^n(N)$ is continuous.
\end{proposition}

\begin{proof}
We will prove our claim by induction. If $A$ is locally uniformly bounded, then it defines a continuous operator on $\Gamma_1^\infty$ with respect to the locally uniform convergence topology, thus $A(\Gamma_1^0(N))\subseteq \Gamma_2^0(N)$ and the case $n=0$ follows.  For the case $n=1$, let $\varphi\in \Gamma_1^1(N)$. Then there is a sequence $(\varphi_m)$ in $\Gamma_1^\infty(N)$ such that $\varphi_m\to \varphi$ and $\nabla_X^1(\varphi_m)\to \nabla_X^1(\varphi)$, where the limits are taking in the locally uniform convergence topology and $X\in\vect{(N)}$. Then 
$$
\nabla_X^2(A\varphi_m)
    =\hat\nabla_X(A)\varphi_m+A\nabla^1_X(\varphi_m).
$$
The case $n=0$ implies that the right hand side converges to 
$ \hat\nabla_X(A)\varphi+A\nabla^1_X(\varphi)$. Thus, by definition $A\varphi\in \Gamma_1^2(N)$ and 
\begin{equation}\label{Leibg}
   \nabla_X^2(A\varphi)=\hat\nabla_X(A)\varphi+A\nabla^1_X(\varphi). 
\end{equation}
Assume our claim holds for $n-1$ and let $A\in \AA^n(\Gamma^\infty_1,\Gamma^\infty_2)$ such that $\hat\nabla_{X_1}\cdots \hat\nabla_{X_k}(A)$ is locally uniformly bounded fields of operators, for every $k\leq n$ and $X_1,\cdots, X_k\in\vect{(N)}$. Then $A\left(\Gamma_1^{n-1}(N)\right)\subseteq \Gamma_2^{n-1}(N)$ and $\hat\nabla_X( A)\left(\Gamma_1^{n-1}(N)\right)\subseteq \Gamma_2^{n-1}(N)$, for every $X\in\vect{(N)}$. Moreover, if $\varphi\in \Gamma_1^{n}(N)$, then equation \eqref{Leibg} implies that  $\nabla_X^2(A\varphi)\in \Gamma_2^{n-1}(N)$. Therefore, $A\varphi\in \Gamma_2^{n}(N)$. The continuity of $A$ follows from noticing that $\nabla^2_{X_1}\cdots\nabla^2_{X_k}(A\varphi)$ is a sum of terms of the form $\hat\nabla_{X_{i_1}}\cdots\hat\nabla_{X_{i_n}}(A)(\nabla^1_{X_{j_1}}\cdots\nabla^1_{X_{j_m}}\varphi)$, where $\{i_1,\cdots,i_n\}\cup\{j_1,\cdots,j_m\}=\{1,\cdots k\}$ and $\{i_1,\cdots,i_n\}\cap\{j_1,\cdots,j_m\}=\emptyset$.

\end{proof}

\begin{Definition}
For each $n\in\mathbb N$, we denote by $\AA^n_{lb}(\Gamma^\infty_1,\Gamma^\infty_2)$ the space formed by all the fields of operators $A\in\AA^n(\Gamma^\infty_1,\Gamma^\infty_2)$ such that $\hat\nabla_{X_1}\cdots \hat\nabla_{X_k}(A)$ is a locally uniformly bounded fields of operators, for every $0\leq k\leq n$ and $X_1,\cdots, X_k\in\vect{(N)}$. Similarly, we denote by $\AA^\infty_{lb}(\Gamma^\infty_1,\Gamma^\infty_2)$ the space formed by the field of operators $A$ such that $A\in \AA^n_{lb}(\Gamma^\infty_1,\Gamma^\infty_2)$, for every $n\in\mathbb N$. We define  $\AA^n_{lb}(\Gamma^\infty_1)=\AA^ n_{lb}(\Gamma^\infty_1,\Gamma^\infty_1)$ and $\AA^\infty_{lb}(\Gamma^\infty_1)=\AA^\infty_{lb}(\Gamma^\infty_1,\Gamma^\infty_1)$. 
\end{Definition}

\section{Hilbert bundles coming from smooth fields of Hilbert spaces and local trivializations}\label{section2}
The following result is well known for finite-dimensional vector bundles, but we could not find it correctly stated in the literature for the infinite-dimensional case. 

\begin{proposition}\label{bundle}
Let $N$ be a Banach manifold and $p:H\to N$ be a field of Hilbert spaces. Let $\{U_i\}_{i\in I}$ be an open cover of $N$ and $\{\Hi_i\}_{i\in I}$ a family of Hilbert spaces, and assume that for each $\lb\in U_i$, there is a unitary operator $T_i(\lb):\Hi(\lb)\to\Hi_i$. Also let $\tau_{ij}:U_i\cap U_j\to \mathcal B(\Hi_i,\Hi_j)$, the map given by $\tau_{i,j}(\lb)=T_j(\lb)T^*_i(\lb)$. The following statements are equivalent.
\begin{enumerate}
    \item[a)] There exist a Banach manifold structure on $H$ such that the open cover $\{U_i\}_{i\in I}$ together with the family of Hilbert spaces $\{\Hi_i\}_{i\in I}$ and the unitary operators $T_i(\lb)$ defines a smooth Hilbert bundle structure on $p:H\to N$.
    \item[b)] Each $\tau_{ij}$ is strongly smooth, i.e. $\tau_{ij}x\in C^\infty(U_i\cap U_j,\Hi_j)$, for each $x\in\Hi_i$.
    \item[c)] Each $\tau_{ij}$ belongs to $\AA^\infty(P^\infty(U_i\cap U_j,\Hi_i),P^\infty(U_i\cap U_j,\Hi_j))$ (see definition of $P^\infty$ in example \ref{exmp}).
 \end{enumerate}
\end{proposition}

The maps $\tau_{ij}$ are called the transition maps of the family $(U_i,\Hi_i,T_i)_{i\in I}$.  
\begin{proof}
By definition a) implies b). Conversely, if b) holds  we claim that the maps $\mathcal T_i:p^{-1}(U_i)\to U_i\times\Hi_i$ define charts on $H$. Indeed, $ \mathcal T_j \mathcal T_i^{-1}(\lb,x)=(\lb,T_j(\lb) T_i^*(\lb)x)=(\lb,\tau_{ij}(\lb)x)$. Since the map $\mathcal{B}(\Hi_i,\Hi_j)\times\Hi_i\ni(A,x)\to Ax\in\mathcal \Hi_j$ is bilinear, each $ \mathcal T_j \mathcal T_i^{-1}$ is smooth if and only if the map $U_i\ni\lb\to \tau_{ij}(\lb)x\in \Hi_j$ is smooth for each $x\in\Hi_i$, i.e. the maps $\mathcal T_i$ defines charts if and only if each $\tau_{ij}$ is strongly smooth. The equivalence between b) and c) follows from recalling that $P^\infty(U_i\cap U_j,\Hi)$ is the space of sections generated by the constant sections (see example \ref{exmp}) and noticing that $\tau_{ij}^*=\tau_{ji}$.  

\end{proof}
\begin{Remark}\label{NL}
{\rm
\begin{enumerate}
    \item[i)] Instead of b), in the literature the definition of Hilbert bundle sometimes requires that each $\tau_{ij}$ is norm smooth. For instance, see \cite{SL}. We thought both concepts were different. However, a referee provided us with a proof demonstrating their equivalence (i.e., strong smoothness implies norm smoothness for operator-valued functions). Nevertheless, the results of this article do not depend on that fact.
\item[ii)] Notice that if $(H\to N,\Gamma^\infty,\nabla)$ is a smooth field of Hilbert spaces and  $T_i(\Gamma^\infty(U_i))=C^\infty(U_i)$, then b) clearly holds. The other properties claimed in Theorem \ref{porpeqivcor} will be shown in the next subsection. The remainder of this subsection aims to develop techniques to achieve at least the inclusion $T_i(\Gamma^\infty(U_i))\subseteq C^\infty(U_i)$.    
\end{enumerate}
}
\end{Remark}

We will provide a forth statement equivalent to a), b) or c) in Proposition \ref{bundle}. We will replace $P^\infty(U_i\cap U_j,\Hi_i)$ in condition c) by any dense smooth submodule of $C^\infty(U_i\cap U_j,\Hi_i)$ and we will require another technical condition in terms of the derivatives of each $\tau_{ij}$ (defined by Proposition \ref{Con}; see also Remark \ref{sder}). 





Let $\Hi_1$ and $\Hi_2$ be Hilbert spaces, and $A(\lb):\Hi_1\to\Hi_2$ be a  bounded operator, for each $\lb\in N$. Clearly, $A$ and $A^*$ are strongly smooth fields of operators if and only if $A\in\AA^\infty(P^\infty(N,\Hi_1),P^\infty(N,\Hi_2))$. In such case, we also have that 






$$
X_1\cdots X_k(A x)=\hat X_1\cdots \hat X_k (A)x,
$$
for every $x\in\Hi_1$ and $X_1,\cdots, X_k\in\vect{(N)}$. Therefore, the domain of each $\hat X_1,\cdots \hat X_k (A)$ is $\Hi_1$, and since it also has a densely defined adjoint, each $\hat X_1,\cdots \hat X_k (A)$ is bounded. Moreover, the previous identity also shows that $\hat X_1,\cdots \hat X_k (A)$ is strongly continuous. The same conclusion follows for $A^*$. In other words,  each $\hat X_1,\cdots \hat X_k (A)$ belongs to $C(N,B(\Hi_1,\Hi_2)_{\ast-st})$, where $B(\Hi_1,\Hi_2)_{\ast-st}$ is the space of bounded operators endowed with the $\ast$-strong topology. The following result asserts that the latter property characterizes strongly smooth fields of operators, even if we replace $P^\infty(N,\Hi_i)$ by any dense smooth $C^\infty(N)$-submodule of $C^\infty(N,\Hi_i)$.  
\begin{proposition}\label{Der}
 Let $\tilde{\Gamma}_1^\infty$ be a dense smooth submodule of $C^\infty(N,\Hi_1)$,  $\tilde{\Gamma}_2^\infty$ be a dense smooth submodule of $C^\infty(N,\Hi_2)$ and $A$ belong to $\AA(\tilde{\Gamma}_1^\infty,\tilde{\Gamma}_2^\infty)$. $A$ is strongly smooth if and only if
 $\hat X_1,\cdots \hat X_k (A)$ belongs to $C(N,B(\Hi_1,\Hi_2)_{\ast-st})$, for every $X_1,\cdots, X_k\in\vect{(N)}$. 
\end{proposition}

\begin{proof}
We have already shown that, if $A$ is strongly smooth field of operators, then each $\hat X_1,\cdots \hat X_k (A)\in C(N,B(\Hi_1,\Hi_2)_{\ast-st})$, for every $X_1,\cdots, X_k\in\vect{(N)}$. Conversely, let $x\in\Hi_1$. Let us show that $Ax$ belongs to $C^1(N,\Hi_2)$. Since $\hat X(A)x$ belongs to $C^0(N,\Hi_2)$, according to  Lemma 5.1.1 \cite{LS} (or appendix \ref{appa}), it is enough to show that the map $\lb\to \langle A(\lb)x,y\rangle$ is smooth and
$$
X(\langle A x,y\rangle)(\lb)=\langle \hat X(A)(\lb)x,y\rangle,
$$
for every $y\in\Hi_2$. Let $f_n \in\tilde\Gamma_2^\infty$ such that $f_n\to y$ (in the canonical topology of $C^\infty(N,\Hi_2)$). In particular, $f_n$ converges locally uniformly to $y$ and $X(f_n)$ converges locally uniformly to $0$, for every $X\in\vect{(N)}$. Since $A x$ is locally uniformly bounded, $\langle A x,f_n\rangle$ converge locally uniformly to $\langle Ax,y\rangle$. Thus, we only need to prove that $X(\langle Ax,f_n\rangle)$ converges locally uniformly to $\langle \hat X(A)x,y\rangle$. Since $A$ is a smooth field of operators,  $\langle A x,f_n\rangle$ is smooth and 
$$
X(\langle Ax,f_n\rangle)=X(\langle x,A^*f_n\rangle)=\langle x,\hat{\overline{X}}(A^*)f_n\rangle+\langle x,A^*\overline{X}(f_n)\rangle=\langle \hat X(A)x,f_n\rangle+\langle Ax,\overline{X}(f_n)\rangle.
$$
Therefore, since $A x$ and $\hat X(A) x$ are locally uniformly bounded, $X(\langle A x,f_n\rangle)$ converge locally uniformly to $\langle \hat X (A)x,y\rangle$. Following the same argument, we can show by induction that $Ax$ belongs to $C^n(N,\Hi_2)$, for every $n\in\mathbb N$, and this completes the proof.
\end{proof}
\begin{Corollary}\label{Xtau} For each $i,j\in I$, let $U_i$, $T_i(\lb):\Hi(\lb)\to\Hi_i$ and $\tau_{ij}$ as in Proposition \ref{bundle}. Also, let $\Gamma^\infty$ be a $C^\infty(N)$-module of sections of $H\to N$ and define $\tilde{\Gamma}_{ij}^\infty:=T_i(\Gamma^\infty|_{U_i\cap U_j})$. Assume that  each $\tilde{\Gamma}_{ij}^\infty$ is a dense smooth subspace in $C^\infty(U_i\cap U_j,\Hi_i)$. The field of operators $\tau_{ij}$ is strongly smooth if and only if
\begin{enumerate}
    \item[d)] $\hat X_1,\cdots \hat X_k (\tau_{ij})$ belongs to $C(U_i\cap U_j,B(\Hi_i,\Hi_j)_{\ast-st})$, for every $X_1,\cdots, X_k\in\vect{(U_i\cap U_j)}$.
    
\end{enumerate}
\end{Corollary}


Let $(H\to N, \Gamma^\infty,\nabla)$ be a smooth field of Hilbert spaces. The notion of smoothness of fields of operators given in Definition \ref{SFO}, and the extension of the connection to such fields of operators, allow us to analyze fields of operators of the form $\hat X(A)$ in Proposition \ref{Der} and its Corollary \ref{Xtau}. However, in order to find further conditions that guarantee that a given family of triples $(U_i,\Hi_i,T_i)_{i\in I}$ defines a Hilbert bundle structure on $H\mapsto N$ (or equivalently, that each $\tau_{ij}$ is strongly smooth), it seems more subtle to require that each $T_i$ defines a smooth field of operators. In other words, we will assume that $T_i(\Gamma^\infty|_{U_i})\subseteq C^\infty(U_i,\Hi_i)$ and that there is a dense smooth $C^\infty(U_i)$-submodule $\tilde\Gamma^\infty_i\subseteq C^\infty(U_i,\Hi_i)$ such that $T^*(\tilde\Gamma^\infty_i)\subseteq \Gamma^\infty(U_i)$. For instance, this occurs when $\tilde{\Gamma}_{i}^\infty:=T_i(\Gamma^\infty|_{U_i})$ is a dense smooth subspace of $C^\infty(U_i,\Hi_i)$ (as in the previous corollary).

Notice that if $\tilde\Gamma^\infty_i$ is a dense smooth subspace of $C^\infty(U_i,\Hi_i)$ and each $T_i$ belongs to $\AA^\infty_{lb}(\Gamma^\infty|_{U_i},\tilde \Gamma_i)$, then each $\tau_{ij}$ is strongly smooth. In fact, Proposition \ref{lub} implies that $T(\Gamma^\infty(U_i))\subseteq C^\infty(U_i,\Hi_i)$ and $T^*(C^\infty(U_i,\Hi_i))\subseteq \Gamma^\infty(U_i)$. Thus, $\tau_{ij}x=T_jT_i^*x$ belongs to $C^\infty(U_i\cap U_j,\Hi_j)$, for every $x\in\Hi_i$. Actually, under those assumptions we will prove that $H\to N$ admits a Hilbert bundle structure with a connection (see Corollary \ref{corisoc}). The following result asserts that, when  $\tilde\Gamma^\infty_i=P^\infty(U_i,\Hi_i)$, we can weaken the condition $T_i\in \AA^\infty_{lb}(\Gamma^\infty|_{U_i},\tilde \Gamma_i)$ and still obtain a Hilbert bundle structure.

\begin{proposition}\label{SF-HB}
Let $(H\to N,\Gamma^\infty,\nabla)$ be a smooth field of Hilbert spaces. Also, let $\{U_i\}_{i\in I}$ be an open cover of $N$ and $\{\Hi_i\}_{i\in I}$ a family of Hilbert spaces, and assume that for each $\lb\in U_i$, there is a unitary operator $T_i(\lb):\Hi(\lb)\to\Hi_i$. 
Also assume that $T_i(\Gamma^\infty|_{U_i})\subseteq C^\infty(U_i,\Hi_i)$ and $T_i^*x$ belongs to $\Gamma^\infty(U_i)$, for every $x\in\Hi_i$.
\begin{enumerate}
    \item[a)] $\hat\nabla_{X_1}\cdots \hat\nabla_{X_k}(T_i)$ is a field of bounded operators, for every $k\in\mathbb{N}$ and $X_1,\cdots, X_k\in\vect{(U_i)}$.
    \item[b)] $T_i(\Gamma^\infty(U_i))\subseteq C^\infty(U_i,\Hi_i)$ if and only if $\hat\nabla_{X_1}\cdots \hat\nabla_{X_k}(T_i)$ maps $\Gamma^\infty(U_i)$ into $C(U_i,\Hi_i)$, for every $k\in\mathbb{N}$ and $X_1,\cdots, X_k\in\vect{(U_i)}$.
    \item[c)] Let $A_{i}(X):=T_i^*\hat\nabla_X(T_i)$.  If  $A_{i}(X)\left(\Gamma^\infty(U_i)\right)\subseteq \Gamma^\infty(U_i)$, then $T_i(\Gamma^\infty(U_i))\subseteq C^\infty(U_i,\Hi_i)$.

\end{enumerate}
\end{proposition}


\begin{proof}
By definition, we have that 
$$
\hat\nabla_{X_1}\cdots \hat\nabla_{X_k}(T_i^*)x=\nabla_{X_1}\cdots \nabla_{X_k}(T_i^*x),
$$
for every $k\in\mathbb{N}$ and $X_1,\cdots X_k\in\vect{(U_i)}$. In particular, the domain $\hat\nabla_{X_1}\cdots \hat\nabla_{X_k}(T_i^*)(\lb)$ is $\Hi_i$, for every $\lb\in U_i$. Proposition \ref{Con} v) and condition iii) in Definition \ref{def} implies that $[\hat\nabla_{X_1}\cdots \hat\nabla_{X_k}(T_i^*)]^*$ is densely defined. The closed graph Theorem and Proposition \ref{Con} v) imply that $\hat\nabla_{X_1}\cdots \hat\nabla_{X_k}(T_i^*)$ is bounded and
$$
\left[\hat\nabla_{X_1}\cdots \hat\nabla_{X_k}(T_i^*)\right]^*=\hat\nabla_{\overline{X_i}}\cdots \hat\nabla_{\overline{X_k}}(T_i), 
$$
for every $k\in\mathbb{N}$ and $X_1,\cdots X_k\in\vect{(U_i)}$. In particular, a) follows.

In order to show b), assume that $\hat\nabla_{X_1}\cdots \hat\nabla_{X_k}(T_i)$ maps $\Gamma^\infty(U_i)$ into $C(U_i,\Hi_i)$, for every $k\in\mathbb{N}$ and $X_1,\cdots, X_k\in\vect{(U_i)}$. It is enough to prove by induction in $n$ that $\hat\nabla_{X_1}\cdots \hat\nabla_{X_k}(T_i)\left(\Gamma^\infty(U_i)\right)\subseteq C^n(U_i,\Hi_i)$, for every $k\in\mathbb{N}\cup\{0\}$, $X_1,\cdots X_k\in \vect{(U_i)}$ and $n\in\mathbb N$. The case $n=0$ and $k\in\mathbb N$ is precisely our initial assumption. Since $T_i$ is unitary, $T_i\left(\Gamma^0(U_i)\right)\subseteq C(U_i,\Hi_i)$ and the case $k=0$ and $n=0$ follows. Assume that $\hat\nabla_{X_1}\cdots \hat\nabla_{X_k}(T_i)\left(\Gamma^\infty(U_i)\right)\subseteq C^{n-1}(U_i,\Hi_i)$, for every $k\in\mathbb{N}\cup\{0\}$, $X_1,\cdots X_k\in \vect{(U_i)}$. If $\varphi\in\Gamma^\infty(U_i)$, then the map $\langle \hat\nabla_{X_1}\cdots \hat\nabla_{X_k}(T_i)\varphi,x\rangle=\langle \varphi,\hat\nabla_{\overline{X_1}}\cdots \hat\nabla_{\overline{X_k}}(T_i^*)x\rangle$ belongs to $C^\infty(U_i)$, for every $x\in\Hi_i$, and we have that
$$
X\left(\langle\hat\nabla_{X_1}\cdots \hat\nabla_{X_k}(T_i)\varphi,x\rangle\right)=\langle\nabla_X(\varphi),\hat\nabla_{\overline{X_1}}\cdots \hat\nabla_{\overline{X_k}}(T_i^*)x\rangle+  \langle\varphi,\nabla_{\overline X}(\hat\nabla_{\overline{X_1}}\cdots \hat\nabla_{\overline{X_k}}(T_i^*)x)\rangle
$$
$$
=\langle\hat\nabla_{X_1}\cdots \hat\nabla_{X_k}(T_i)\nabla_X(\varphi)+\hat\nabla_X\hat\nabla_{X_1}\cdots \hat\nabla_{X_k}(T_i)\varphi,x\rangle.
$$
Since $\hat\nabla_{X_1}\cdots \hat\nabla_{X_k}(T_i)\nabla_X(\varphi)+\hat\nabla_X\hat\nabla_{X_1}\cdots \hat\nabla_{X_k}(T_i)\varphi\in C^{n-1}(U_i,\Hi_i)$,  \cite[Lemma 5.1.1 ]{LS} (or appendix \ref{appa}) implies that, for each $k\in\mathbb{N}\cup\{0\}$ and $X_1,\cdots X_k\in \vect{(U_i)}$, $\hat\nabla_{X_1}\cdots \hat\nabla_{X_k}(T_i)\left(\Gamma^\infty(U_i)\right)\subseteq C^{n}(U_i,\Hi_i)$. Therefore, $T_i(\Gamma^\infty(U_i))\subseteq C^\infty(U_i,\Hi_i)$. The converse statement in b) is trivial.

We follow a similar but simpler argument to show c). Assume that $A_{i}(X)(\Gamma^\infty(U_i))\subseteq \Gamma^\infty(U_i)$. It is enough to prove by induction that $T_i\left(\Gamma^\infty(U_i)\right)\subseteq C^n(U_i,\Hi_i)$, for every $n\in\mathbb N$. We showed the case $n=0$ during the proof of b). Assume that $T_i\left(\Gamma^\infty(U_i)\right)\subseteq C^{n-1}(U_i,\Hi_i)$. If $\varphi\in\Gamma^\infty(U_i)$, then the map $\langle T_i\varphi,x\rangle=\langle \varphi,T_i^*x\rangle$ belongs to $C^\infty(U_i)$, for every $x\in\Hi_i$, and we have that 
$$
X(\langle T_i\varphi,x\rangle)=\langle \nabla_X(\varphi),T^*_ix\rangle+\langle \varphi,\nabla_{\overline{X}}(T_i^* x)\rangle=\langle T_i \nabla_X(\varphi)+ \hat\nabla_X(T_i)\varphi,x\rangle=\langle T_i(\nabla_X(\varphi)+ A_{i}(X)(\varphi)),x\rangle
$$
Since $ \nabla_X(\varphi)+ A_{i}(X)(\varphi)\in C^{n-1}(U_i,\Hi_i)$, Lemma 5.1.1  \cite{LS} (or appendix \ref{appa}) implies that $T_i\left(\Gamma^\infty(U_i)\right)\subseteq C^{n}(U_i,\Hi_i)$. Therefore, $T_i(\Gamma^\infty(U_i))\subseteq C^\infty(U_i,\Hi_i)$.
\end{proof}
According to the discussion and results so far in this section, we propose the following notion of  trivialization, which is meant to insure the existence of a Hilbert bundle structure such that $\Gamma^\infty(N)\subseteq\Gamma^\infty(N,H)$. 
\begin{Definition}\label{wLT}
Let $(H\to N,\Gamma^\infty,\nabla)$ be a field of Hilbert spaces. Also let $\{U_i\}_{i\in I}$ be an open cover of $N$ and $\{\Hi_i\}_{i\in I}$ a family of Hilbert spaces, and assume that for each $\lb\in U_i$, there is a unitary operator $T_i(\lb):\Hi(\lb)\to\Hi_i$. We say that $(U_i,\Hi_i,T_i)_{i\in I}$ is a weak smooth local trivialization of $(H\to N,\Gamma^\infty,\nabla)$ if for each $i\in I$ and $x\in\Hi_i$, $T_i^*x$ belongs to $\Gamma^\infty(U_i)$ and $T_i\left(\Gamma^\infty(U_i)\right)\subseteq C^\infty(U_i,\Hi_i)$.
\end{Definition}
It is also important to be able to determine when two different smooth fields of Hilbert spaces induce the same Hilbert bundle structure. The proof of the following result is straightforward. 

\begin{proposition}\label{semiunique}
Let $(\Gamma^\infty_1,\nabla^1)$ and $(\Gamma^\infty_2,\nabla^2)$ be two smooth field of Hilbert spaces structures on $H\mapsto N$ admitting weak smooth local trivializations. If $\Gamma^\infty_1(N)=\Gamma_2^\infty(N)$, then the corresponding Hilbert bundles structures are equivalent.  
\end{proposition}

\subsection{Hilbert bundles with a connection}\label{Hb}
We seek to establish conditions under which $H\to N$  admits a Hilbert bundle structure equipped with a connection. Let $(H\to N, \Gamma^\infty,\nabla)$ be a smooth field of Hilbert spaces. Also, let $\{U_i\}_{i\in I}$ be an open cover of $N$ and $\{\Hi_i\}_{i\in I}$ a family of Hilbert spaces, and assume that for each $\lb\in U_i$, there is a unitary operator $T_i(\lb):\Hi(\lb)\to\Hi_i$. For simplicity, assume that $\tilde \Gamma^\infty_i=T_i(\Gamma^\infty|_{U_i})$ is a dense smooth subspace of $ C^\infty(U_i,\Hi_i)$. Let us introduce the smooth field of operators $\tilde{A}_{i}(X):=-\hat\nabla_X(T_i)T_i^*:\tilde{\Gamma}^\infty_i\to \tilde \Gamma^\infty_i$.  Proposition \ref{Con} and some straightforward computations show the following identities:


$$
T_i\nabla_X=X T_i+ \tilde{A}_{i}(X) T_i.
$$
\begin{equation}\label{Atau}
  \hat X(\tau_{ij})f=(\tau_{ij}\tilde{A}_{i}(X)-\tilde{A}_{i}(X)\tau_{ij})f
\end{equation}
\begin{equation}\label{opcon}
   \tilde{A}_{i}(X+Y)=\tilde{A}_{i}(X) +\tilde{A}_{i}(Y),\qquad  \tilde{A}_{i}(aX)=a\tilde{A}_{i}(X). 
\end{equation}
\begin{equation}\label{derin}
   \langle \tilde{A}_{i}(X) f,g\rangle=\langle f,-\tilde{A}_{i}(\overline{X})g\rangle,  
\end{equation}
for every $f,g\in\tilde{\Gamma}^\infty_i$, $a\in C^\infty(U_i)$ and $X,Y\in\vect{(U_i)}$. Notice that if we define $A_{i}(X)=T_i^*\hat\nabla_X(T_i)$ (as in Proposition \ref{SF-HB}), then $A_{i}(X)=-T_i^*\tilde{A}_{i}(X)T_i$ on $\Gamma^\infty|_{U_i}$. Moreover, the family of fields of operators $A_{i}(X)$ also satisfies equations \eqref{opcon} and \eqref{derin} replacing $f,g$ for $\varphi,\psi\in\Gamma^\infty|_{U_i}$. 

Without assuming that $T_i(\Gamma^\infty|_{U_i})$ is a dense smooth subspace of $ C^\infty(U_i,\Hi_i)$, and instead we require that there is a dense smooth submodule $\tilde\Gamma^\infty_i$ of $ C^\infty(U_i,\Hi_i)$ such that $T_i\in \AA^\infty(\Gamma^\infty|_{U_i}, \tilde\Gamma^\infty_i)$, then to define $\tilde{A}_i(X)$ we would also need to assume that $T_i(\Gamma^\infty(U_i))\subseteq C^\infty(U_i,\Hi_i)$. 

\begin{proposition}
Let $(H\to N,\Gamma^\infty,\nabla)$ be a smooth field of Hilbert spaces. Also, let $\{U_i\}_{i\in I}$, $\{\Hi_i\}_{i\in I}$ and $T_i(\lb):\Hi(\lb)\to\Hi_i$ define Hilbert bundle structure on $H\to N$ such that $\Gamma^\infty(U_i)\subseteq\Gamma^\infty(U_i,H|_{U_i})$. Assume that there is a dense smooth subspace $\tilde\Gamma^\infty_i\subseteq C^\infty(U_i,\Hi_i)$ such that $T_i\in \AA^\infty(\Gamma^\infty|_{U_i}, \tilde\Gamma^\infty_i)$, for each $i\in I$.  The following statement are equivalent:
\begin{enumerate}
    \item[a)] $H\to N$ admits a Hilbert bundle structure with connection $\tilde\nabla$ such that $\tilde\nabla|_{\Gamma^\infty}=\nabla$.
    \item[b)] $\tilde{A}_{i}(X)=-\hat\nabla_X(T_i)T_i^*:\tilde{\Gamma}_i^\infty\mapsto C^\infty(U_i,\Hi_i)$ is a field of bounded operators such that $\tilde{A}_{i}(X)\left[C^\infty(U_i,\Hi_i)\right]\subseteq C^\infty(U_i,\Hi_i)$ and equations \eqref{Atau}, \eqref{opcon} and \eqref{derin} hold true, for every $i\in I$, $a\in C^\infty(U_i)$, $f,g\in C^\infty(U_i,\Hi_i)$ and $X,Y\in\vect{(U_i)}$.
\end{enumerate}
\end{proposition}
\begin{proof}
If we assume a), then b) follows from taking $\Gamma^\infty=\Gamma^\infty(N,H)$ in the definition of $\tilde{A}_{i}(X)$ (so $\tilde{\Gamma}^\infty_i=C^\infty(U_i,\Hi_i)$). Conversely, let us assume b). Since $H|_{U_i}$ is diffeomorphic with $U_i\times \Hi_i$, we have that $T_i(\Gamma^\infty(U_i,H|_{U_i}))=C^\infty(U_i,\Hi_i)$, hence $T_i(\Gamma^\infty(U_i))\subseteq C^\infty(U_i,\Hi_i)$ and then $\tilde A_i(X)$ is well defined. Moreover, for $\varphi\in \Gamma^\infty(N,H)$ we can define
$$
\tilde\nabla_X(\varphi)|_{U_i}:=T_i^*(X+\tilde{A}_{i}(X))T_i(\varphi|_{U_i}).
$$
Equation \eqref{Atau} implies that $\tilde\nabla_X(\varphi)$ is well defined on $N$. Equations \eqref{opcon} and \eqref{derin} imply conditions i) and ii) in the definition of a connection. Thus, $\tilde\nabla$ defines the required connection and this finishes the proof.
\end{proof}

\begin{proof}[Proof of Theorem \ref{porpeqivcor}]
If $T_i(\Gamma^\infty(U_i))=C^\infty(U_i,\Hi_i)$, then clearly $\tau_{ij}$ is strongly smooth and condition b) of Proposition \ref{bundle} implies that $(U_i,\Hi_i,T_i)$ defines a Hilbert bundle structure on $H\to N$. Moreover, by construction $H|_{U_i}$ is diffeomorphic to $U_i\times\Hi_i$; then $T_i(\Gamma^\infty(U_i,H|_{U_i}))=C^\infty(U_i,\Hi_i)$. Since $T_i$ is bijective,  $\Gamma^\infty(U_i,H|_{U_i})=\Gamma^\infty(U_i)$. Clearly condition b) in the previous Proposition holds, therefore $(U_i,\Hi_i,T_i)$ defines a Hilbert bundle structure on $H\to N$ with a connection satisfying the required conditions. Conversely, if $(U_i,\Hi_i,T_i)$ defines a Hilbert bundle structure on $H\to N$ with a connection satisfying the required conditions, then  $T_i(\Gamma^\infty(U_i,H|_{U_i}))=C^\infty(U_i,\Hi_i)$ and therefore $T_i(\Gamma^\infty(U_i))=C^\infty(U_i,\Hi_i)$. Assume that any of the previous equivalent conditions hold. If $\varphi\in\Gamma^\infty(N)$, then clearly $\varphi|_{U_i}\in\Gamma^\infty(U_i)=\Gamma^\infty(U_i,H|_{U_i})$, hence $\varphi\in\Gamma^\infty(N,H)$. Similarly, let $\varphi\in\Gamma^\infty(N,H)$ and $\{e_i\}_{i\in I}$ a partition of unity subordinate to $\{U_i\}_{i\in I}$. Then $\varphi|_{U_i}\in\Gamma^\infty(U_i)$ and there is a sequence $\varphi_{k}^{i}$ such that $\varphi_k^{i}\to\varphi|_{U_i}$ uniformly on each compact subset of $U_i$. By definition of a partition of unity, the sequence $\varphi_k=\sum_i e_i\varphi^{i}_k$ is well defined and $\varphi_k\to \varphi$. Thus, $\varphi\in\Gamma^0(N)$. The same argument shows that $\varphi\in\Gamma^1(N)$. By induction is straightforward to show that $\varphi\in\Gamma^n(N)$, for every $n\in\mathbb N$. Therefore, $\Gamma^\infty(N)=\Gamma^\infty(N,H)$.            
\end{proof}


\begin{Corollary}\label{corisoc}
Let $(H\to N,\Gamma^\infty,\nabla)$ be a smooth field of Hilbert spaces. Also let $\{U_i\}_{i\in I}$ be an open cover of $N$ and $\{\Hi_i\}_{i\in I}$ a family of Hilbert spaces, and assume that for each $\lb\in U_i$, there is a unitary operator $T_i(\lb):\Hi(\lb)\to\Hi_i$. Also assume that there is a dense smooth subspace $\tilde\Gamma^\infty_i\subseteq C^\infty(U_i,\Hi_i)$ such that $T_i\in\AA^\infty(\Gamma^\infty|_{U_i},\tilde\Gamma^\infty_i)$, for each $i\in I$. If  $T_i\in \AA^\infty_{lb}(\Gamma^\infty|_{U_i},\tilde{\Gamma}^\infty_i)$, then $T_i$ defines an isomorphism of Fr\'echet spaces between $\Gamma^\infty(U_i)$ and $C^\infty(U_i,\Hi_i)$. In particular, in such case $H\mapsto N$ admits a Hilbert bundle structure with a connection and $\Gamma^\infty(N)=\Gamma^\infty(N,H)$. If $\tilde\Gamma^\infty_i=T_i(\Gamma^\infty|_{U_i})$ is a dense smooth subspace of $ C^\infty(U_i,\Hi_i)$ and $\tilde A^{i}_X\in \AA^\infty_{lb}(\tilde\Gamma^\infty_i)$ for every $X\in\vect{(U_i)}$, then $T_i\in \AA^\infty_{lb}(\Gamma^\infty|_{U_i},\tilde{\Gamma}^\infty_i)$.
\end{Corollary}
\begin{proof}
By definition, the density of $\tilde{\Gamma}_{i}^\infty$ implies that $\tilde{\Gamma}_{i}^\infty(U_i)=C^\infty(U_i,\Hi_i)$. Therefore, Proposition \ref{lub} implies that $T\left(\Gamma^\infty(U_i)\right)\subseteq C^\infty(U_i,\Hi_i)$ and $T^*\left(C^\infty(U_i,\Hi_i)\right)\subseteq \Gamma^\infty(U_i)$ (continuously). Since $T T^*=I$, it follows that $T\left(\Gamma^\infty(U_i)\right)=C^\infty(U_i,\Hi_i)$ (and $T^*\left(C^\infty(U_i,\Hi_i)\right)=\Gamma^\infty(U_i)$).

Notice that $\hat\nabla_X(T_i)=-\tilde{A}_{i}(X) T_i$. Thus, induction and Proposition \ref{Leib1} finish our proof.
\end{proof}

Based on the discussion and results presented thus far in this section, we propose the following notion of smooth trivialization, which essentially follows from Theorem \ref{porpeqivcor} and is intended to ensure the existence of a suitable Hilbert bundle structure with a connection. 
\begin{Definition}\label{LT}
Let $(H\to N,\Gamma^\infty,\nabla)$ be a field of Hilbert spaces. Also, let $\{U_i\}_{i\in I}$ be an open cover of $N$ and $\{\Hi_i\}_{i\in I}$ a family of Hilbert spaces, and assume that for each $\lb\in U_i$, there is a unitary operator $T_i(\lb):\Hi(\lb)\to\Hi_i$. We say that $(U_i,\Hi_i,T_i)_{i\in I}$ is a smooth local trivialization of $(H\to N,\Gamma^\infty,\nabla)$ if for each $i\in I$,  $T_i(\Gamma^\infty(U_i))=C^\infty(U_i,\Hi_i)$.

\end{Definition}

Let $(\Gamma^\infty_1,\nabla^1)$ and $(\Gamma^\infty_2,\nabla^2)$ be two smooth field of Hilbert spaces structures on $H\mapsto N$ admitting smooth local trivializations. It follows directly from Proposition \ref{semiunique} and Theorem \ref{porpeqivcor} that $\Gamma^\infty_1(N)=\Gamma_2^\infty(N)$ if and only if the corresponding Hilbert bundles structures are equivalent.


The rest of this subsection will concern the construction of trivializations from a given smooth field of Hilbert spaces. However, such construction will not be required in the family of examples studied in the next section (in other words, the unitary maps $T_i(\lb)$ will be inherent in the framework, so we will not need to construct them). 

The problem of constructing trivilizations was considered only globally in \cite{LS}. Indeed, one of their main results asserts that if $(H\to N,\Gamma^\infty,\nabla)$ is analytic and flat, then there is a (global) trivialization $T:H\to V$ such that $T(\Gamma^\infty)\subseteq C^\infty(N,V)$ and $XT=T\nabla_X$, where $N$ is required to be a connected and simply connected analytic manifold, and $V$ is a suitable Hilbert space (see Theorem 5.1.2 in \cite{LS}). The most difficult part in the proof is to show that analiticity implies that for each $\lb\in N$, there exists an open neighborhood $U$ such that through every point in $H|_U$ there passes a horizontal section $\varphi\in\Gamma^\infty(U)$ \cite[Lemma 4.2.1]{LS}. Since $N$ was assumed to be simply connected, the latter property holds globally \cite[Lemma 4.1.3]{LS}. In our case, in order to obtain a Hilbert bundle, we will need the existence of horizontal sections, but we do not need to assume that $N$ is simply connected. Essentially, we will extend the proof of Theorem 5.12 in \cite{LS} taking into account the additional difficulties coming from local trivializations and non-flatness.  
\begin{Theorem}\label{HS}

Let $(H\to N,\Gamma^\infty,\nabla)$ be a smooth field of Hilbert spaces, where $N$ is a connected finite-dimensional manifold. Assume that for every $\lb\in N$ there is an open neighborhood $U$ and a family of $\Gamma^\infty(U)$- smooth field of operators $\{A(X)\}_{X\in\vect{(U)}}$ such that:
\begin{enumerate}
    \item[a)] Equations \eqref{opcon} and \eqref{derin} holds true.
    \item[b)] $A(X):\Gamma^\infty(U)\mapsto\Gamma^\infty(U)$ is continuous with respect the topology induced by $\Gamma^n(U)$, for every $X\in\vect{(U)}$ and $n\in\mathbb{N}$.  
    \item[c)] Through every point in $H|_U$ there passes a section $\varphi\in\Gamma^\infty(U)$ such that  $\nabla_X\varphi=A(X)\varphi$, for every $X\in\vect{(U)}$.
\end{enumerate}
Then $(H\to N,\Gamma^\infty,\nabla)$ admits a smooth local trivialization.

\end{Theorem}

We say that a section $\varphi$ is horizontal with respect to the family of field of operators $\{A(X)\}$ if the identity $\nabla_X\varphi=A(X)\varphi$ holds for every $X\in\vect{(U)}$. When $A(X)=0$ for every $X\in\vect{(U)}$, we recover the notion of horizontal section given in \cite{LS}; thus our result generalizes the result given in \cite{LS} considering local trivializations and also considering the case $A(X)\neq 0$. 
\begin{proof}

Let $\{U_i\}_{i\in I}$ be an open cover of $N$ and $\{A_{i}(X)\}_{X\in\vect{(U_i)}}$ be a family of smooth fields of operators satisfying a) b) and c) for each $i\in I$. Also let 
$$
\Hi_i=\{\varphi\in\Gamma^\infty(U_i)\mid \varphi\,\,\text{is horizontal with respect to} \,\{A_{i}(X)\}_{X\in\vect{(U_i)}}\}$$ 
Following the proof of Lemma 4.1.1 in \cite{LS}, notice that, if $\varphi,\psi\in\Hi_i$
$$
X\left(\langle\varphi,\psi\rangle\right)=\langle\nabla_X\varphi,\psi\rangle+\langle\varphi,\nabla_{\overline{X}}\psi\rangle=\langle A_{i}(X)\varphi,\psi\rangle+\langle\varphi,A_{i}(\overline{X})\psi\rangle=0.
$$
Since $N$ is connected, the map $\langle\varphi,\psi\rangle$ is constant on $N$, for every $\varphi,\psi\in\Hi_i$, thus it defines an inner product on $\Hi_i$. Moreover, since through every point in $H|_{U_i}$ there passes a horizontal section, the map $\tilde T_i(\lb):\Hi_i\mapsto \Hi(\lb)$ given by $\tilde T_i(\lb)\varphi=\varphi(\lb)$ is onto. In particular, $\Hi_i$ is a Hilbert space and 
$\tilde T_i(\lb)$ is unitary. Let us show that if we take $T_i(\lb):=\tilde T_i^*(\lb)$, then $(U_i,\Hi_i,T_i)_{i\in I}$ is a smooth local trivialization. For simplicity, whenever we interpret a horizontal section $\varphi$ as a point in the Hilbert space $\Hi_i$ we will denote it by $\hat\varphi$. By definition, if $\varphi$ is a horizontal section in $\Gamma^\infty(U_i)$, then $T_i\varphi=\hat\varphi$. Let us prove that $T_i(\Gamma^\infty(U_i))\subseteq C^\infty(U_i,\Hi_i)$. Let $\psi\in\Gamma^\infty(U_i)$. Since each $T_i(\lb)$ is unitary, clearly $T_i\psi\in C(U_i,\Hi_i)$ (in fact, $T_i(\Gamma^0(U_i))\subseteq C(U_i,\Hi_i)$). Assume that $T_i\psi\in C^n(U_i,\Hi_i)$. For any $\hat\varphi\in\Hi_i$ and $X\in\vect (U_i)$, we have that $\langle T_i\psi,\hat\varphi\rangle=\langle \psi,\varphi\rangle$ belongs to $C^{n+1}(U_i)$ and 
$$
X\langle T_i\psi,\hat\varphi\rangle=X\langle \psi,\varphi\rangle=\langle \nabla_X(\psi),\varphi\rangle+\langle \psi,\nabla_{\overline{X}}(\varphi)\rangle=\langle T_i\nabla_X(\psi),\hat\varphi\rangle+\langle \psi,A(\overline{X})\varphi\rangle
$$
$$
=\langle T_i(\nabla_X- A(X))\psi,\hat\varphi\rangle
$$
Since $T(\nabla_X- A(X))\psi$ belongs to $C^n(U_i,\Hi_i)$, Lemma 5.1.1 in \cite{LS} (or appendix \ref{appa}) implies that $T_i\psi\in C^{n+1}(U_i,\Hi_i)$ and
$X(T_i\psi)=T_i(\nabla_X- A(X))\psi$.

Let $\tilde\Gamma_i$ the space of sections of the form $\sum^m a_j\hat\varphi_j$, with $a_j\in C^\infty(U_i)$ and $\hat\varphi_j\in\Hi_i$. Clearly, $T_i^*(\tilde\Gamma_i)\subseteq \Gamma^\infty(U_i)$. Let us show that each $T_i^*$ is continuous. We will prove by induction on $n$ that if $(f_k)$ is a sequence such that $f_k\to f$ in $\tilde\Gamma_i$ with the $C^n$-topology then $T_i^*(f_k)\to T_i^*(f)$ with the $\Gamma^n(U_i)$-topology. Since $T_i^*$ is unitary, the claim follows trivially for $n=0$. For $n\geq 1$, assume the claim holds true for $n-1$. It is enough to show that $T_i^*(f_k)\to T_i^*(f)$ in $\Gamma^{n-1}(U_i)$   and $\nabla_X[T_i^*(f_k)]\to \nabla_X[T_i^*(f)]$ in $\Gamma^{n-1}(U_i)$, for every $X\in\vect{(U_i)}$. Since $f_k\to f$ in $C^{n-1}(U_i,\Hi_i)$, the first part follows from the inductive hypothesis. Moreover, b) implies that $A(X) T_i^*(f_k)\to A(X) T_i^*(f)$ in $\Gamma^{n-1}(U_i)$. Therefore, since $X(f_k)\to X(f)$ in $C^{n-1}(U_i,\Hi_i)$, we have that 
$$
\nabla_X(T_i^*f_k)=T_i^*X(f_k)- A(X) T^*_if_k\;\longrightarrow \;T_i^*X(f)- A(X) T^*_if=\nabla_X(T_i^* f).
$$
Since $\tilde\Gamma^\infty_i$ is dense in $C^\infty(U_i;\Hi_i)$, we have that  $T_i^*\left[ C^\infty(U_i,\Hi_i)\right]\subseteq \Gamma^\infty(U_i)$. Moreover, if $f\in C^\infty(U_i,\Hi_i)$, then $f=T_i(T_i^*f)\in T_i(\Gamma^\infty(U_i))$. Therefore $T_i(\Gamma^\infty(U_i))=C^\infty(U_i,\Hi_i)$ and this finishes the proof. 

\end{proof}
\begin{Corollary}
If $(H\to N,\Gamma^\infty,\nabla)$ is a flat analytic field of Hilbert spaces, then it admits a full local trivialization $(U_i,\Hi_i,T_i)_{i\in I}$ such that $A_{i}=0$.
\end{Corollary}

We finish this section by providing a way to construct trivializations.
Let $(H_1\to N,\Gamma_1^\infty,\nabla^1)$ and $(H_2\to N,\Gamma_2^\infty,\nabla^2)$ be smooth fields of Hilbert spaces and $S:\Gamma_1^\infty\mapsto \Gamma_2^\infty(N)$ be a smooth field of unitary operators. Assume that $(H_2\to N,\Gamma_2^\infty,\nabla^2)$ admits a local trivialization $(U_i,\Hi_i, T^2_i)$. Then $T_i^1=T_i^2 S$ defines a local trivialization of $(H_1\to N,\Gamma_1^\infty,\nabla^1)$ and 
$$
\tilde{A}^1_{i}(X)= \tilde{A}^2_{i}(X)+ T^2_i S\hat\nabla_X(S^*)(T^2_i)^*.
$$
For instance, if $(H_1\to N,\Gamma_1^\infty,\nabla^1)$ is a projectively flat smooth field of Hilbert spaces and the curvature is exact, then one can construct a flat smooth field of Hilbert spaces $(H_2\to N,\Gamma_2^\infty,\nabla^2)$ and a smooth field of operators $S:\Gamma_1^\infty\mapsto\Gamma_2^\infty$, as explained in subsection I.2.4 in \cite{LS}. In particular, $(H_1\to N,\Gamma_1^\infty,\nabla^1)$ admits a local trivialization if $(H_1\to N,\Gamma_1^\infty,\nabla^1)$ is also analytic (see Theorem 2.4.2 in \cite{LS}). 
\section{Riemannian direct images}\label{RDI}
Let $M$ and $N$ be oriented Riemannian manifolds with dimension $m$ and $k$ respectively and $k<m$.   
Let $\rho:M\to N$ be a smooth submersion. The implicit function theorem guarantees that $M_\lb:=\rho^{-1}(\lb)$ is a $(m-k)$-submanifold of $M$, for each $\lb\in N$. 

Recall that, by definition $D\rho(x):T_x M\to T_{\rho(x)}N$ is an epimorphism, for each $x\in M$. Also, $\text{Ker}D\rho(x)=T_x M_{\rho(x)}$, then the restriction of $D\rho(x)$ to $T_x^{\perp}M_{\rho(x)}$ defines an isomorphism. 
Given this fact, the following notation will be useful.

 \begin{Definition}\label{xjmu}
     Let $\rho:M\to N$ be a submersion, and let $\eta_\lb$ be the Riemannian volume form on $M_\lb$. 
     \begin{enumerate}
         \item[a)] For $X\in\vect(N)$, we denote by $\check X$ the only vector field on $M$ normal to each $M_\lb$ such that $D\rho(\check X)=X$, in other words 
$$
\check X(q)=(D\rho(q)|_{T
^\perp_q M_{\rho(q)}})^{-1}(X(\rho(q)))\,,\quad q\in M.
$$
\item[b)] We denote $J(q):=J_\rho(q):=\det[D\rho|_{T_q^\perp M_{\rho(q)}}]$.
\item[c)] We define the volume form $\mu_\lb=J^{-1}\eta_\lb$. 
     \end{enumerate}
 \end{Definition}

The implicit function theorem guarantees that $\check X$ is a smooth vector field. We interpret $\check X$ as the natural lift of $X$ from $N$ to $M$. Notice that if $N=\R$, then $ J_\rho(x)=||\nabla\rho(x)||$.

Let $\pi:E\to M$ be a finite-dimensional Hermitian vector bundle, with fiber $F$ and a Hermitian connection $\nabla^E$. Denote by $\|\cdot\|_x$ the norm on $\pi^{-1}(x)$, for each $x\in M$. Also, denote by $\Gamma(E)$ the corresponding space of sections and by $\Gamma^{\infty}(E)$ the space of smooth sections. Notice that $E_\lb:=\pi|_{E_\lb}:E_\lb\to M_\lb$ is also a Hermitian bundle, where $E_\lb=\pi^{-1}(M_\lb)$. We will endow the field of Hilbert spaces 
\begin{equation}\label{hilbertexample}
\Hi(\lb)=L^2(E_\lb):=\Big\{\varphi\in\Gamma(E_\lb) \,\big| \,\varphi\text{ is measurable}\, , \int_{M_\lb}\|\varphi(x)\|^2_x d\mu_\lb(x)<\infty\Big\}
\end{equation}
with an explicit smooth structure. We call such a field of Hilbert spaces a Riemannian direct image. We will show that such a field admits a smooth trivialization if $\rho$ defines a fiber bundle and a suitable condition on $\nabla^{E}$ is required (which becomes trivial when $\rho$ proper map). Within the holomorphic framework (i.e., $M,N$ are complex manifolds and $\rho$ is holomorphic), such a problem was considered in Chapter II of \cite{LS}. However, the construction there cannot be adapted to the smooth framework, as explained in Subsection 6.5 of \cite{LS}. In fact, the authors of \cite{LS}, motivated by some geometric quantization problems, suggested that in order to overcome the latter issue, we might also require an Ereshmann's connection on $M$. In our case, the canonical Ereshmann's connection  $M\ni x\to T_x
^\perp M_{\rho(x)}$ will turn to be a fundamental ingredient in the construction of the smooth structure of our field of Hilbert spaces. In subsection \ref{Hcase}, we will discuss in more detail the similarities and differences between Riemnanian direct images and holomorphic direct images.

\subsection{Derivating integrals over $M_\lb$ with respect to $\lb$.}\label{dintlb}

Since the inner products on each $\Hi(\lb)$ are an integral over $M_\lb$ and we are looking for a connection satisfying condition \emph{ii)} in Definition \ref{def}, we would like to compute derivatives of functions of the form 
\begin{equation}\label{f1}
F(\lb)=\int_{M_\lb}\!f\,d\mu_{\lb}\,,\quad f\in  C^\infty(M)\,,\; \lb\in N.    
\end{equation}
Consider the case
 $N=\R^k$. Thus, $\rho=\rho_1\times\cdots\times\rho_k$ and we have that $T_x^{\perp}M_{\rho(x)}=\text{span}\langle\nabla\rho_i(x)\mid i=1,\cdots, k\rangle$, where $\rho_i\in C^\infty(M)$. If $\lb=\lb_1\times\cdots\times\lb_k\in\R^k$,  it will become useful to consider the submanifolds
\begin{equation}\label{mjlb}
M_\lb^i:=\{x\in M\mid\rho_n(x)=\lb_n,\forall n\neq i\}\,,\quad i=1,\dots k.    
\end{equation}
Clearly $M_\lb$ is a submanifold of $M_\lb^i$ of co-dimension $1$. Let also $J_i(x)=J_{\tilde\rho_i}(x)$, where $\tilde\rho_i:=\rho_1\times\cdots\times\widehat\rho_i\times\cdots\times\rho_k$, i.e., $\tilde\rho_i$ is obtained by removing $\rho_i$ from $\rho$. In particular, $M_\lb^i=\tilde\rho_i^{-1}(\lb)$. 

\begin{Lemma}\label{l1}
Let $M$ be a Riemannian manifold and $\rho_i\in C^\infty(M)$, with $1\leq i\leq k$ and $k<m$.  Assume that $\rho=\rho_1\times\cdots\times\rho_k$ is submersion and let $\pi_x^i$ be the orthogonal projection of $T_x M$ onto the subspace $\langle \nabla\rho_n(x)\mid i\neq n\rangle^\perp$. 
\begin{enumerate}
    \item[a)] For each $i=1,\dots,k$, the vector field $\frac{\check\partial}{\partial\lb_i}$ is given by
\begin{equation}\label{hat}
\check{\frac{\partial}{\partial\lb_i}}(x)=\frac{1}{||\pi_x^i(\nabla\rho_i(x))||^2}\pi_x^i(\nabla\rho_i(x)).    
\end{equation}

\item[b)] The gradient of the restriction of $\rho_i$ to $M_\lb^i$ is $\pi^i(\nabla\rho_i)$. In particular, for each $x\in M$, we have that $J(x)=J_i(x)||\pi_x^i(\nabla\rho_i(x))||$.
\end{enumerate}

\end{Lemma}
\begin{proof}
Let $v$ be the only vector orthogonal to each $M_\lb$ such that $D\rho(v)=\frac{\partial}{\partial\lb_i}$. Thus, $D\rho(v)(\lb_n)=\delta_{in}$, for each $n=1,\cdots, k$. By definition, $D\rho(v)(\lb_n)=\langle\nabla\rho_n,v\rangle$. It is straightforward to check that the right-hand side of \eqref{hat} satisfies the required condition.

For the second part of our result, if $x\in M_{\lb}$, then $\nabla (\rho_i|_{M^i_\lb})(x)\in T_x M_\lb^{\perp}\cap \langle\nabla\rho_n(x)\mid i\neq n\rangle^\perp$. Thus, $\nabla (\rho_i|_{M^i_\lb})(x)=C\pi_x^i(\nabla\rho_i(x))$ for some real constant $C$. Moreover, 
$$
\langle\nabla (\rho_i|_{M^i_\lb})(x),\pi_x^i(\nabla\rho_i(x))\rangle=D\rho(x)[\pi_x^i(\nabla\rho_i(x))](\lb_i)=||\pi_x^i(\nabla\rho_i(x))||^2.
$$
Therefore $C=1$. The last claim follows after recalling that $J(x)=|\det(D\rho(x)|_{T_x^\perp M_\lb})|$ and using the orthogonal decomposition $T_x^\perp M_\lb=\text{span}\{||\pi_x^i(\nabla\rho_i(x))||\frac{\check\partial}{\partial\lb_i}\}\oplus T_x^\perp M^i_\lb$.
\end{proof}

In order to compute derivatives of $F$ defined in \eqref{f1}, we will need to consider the notion of divergence of a vector field. 
Let $Y$ be a vector field on $M$ and $\eta$ a volume form on $M$. By definition, the divergence of $Y$ with respect to $\eta$ is the unique smooth function $\div_\eta(Y)\in C^\infty(M)$ such that
$$
\mathcal{L}_Y(\eta)=\div_\eta(Y)\eta,
$$
where $\mathcal{L}$ is the Lie derivative on $M$. If $\eta$ is the canonical volume form on $M$ (coming from its Riemannian structure), then we omit $\eta$ in the notation, i.e. we write $\div(Y)$. In \cite{LS}, it was considered an extension of the previous definition of the divergence suitable for our framework. Let $\nu$ be a $(m-k)$-form such that its restriction to each $M_\lb$ is a volume form. The divergence of $Y$ with respect to $\nu$ is the unique smooth function $\div_\nu(Y)$ such that 
$$
\mathcal{L}_Y(\nu)|_{M_\lb}=\div_\nu(Y)\nu|_{M_\lb}, \forall \lb\in N
$$
The following Lemma provides some identities concerning the computation of the divergence of vector fields with respect to different forms. 
\begin{Lemma}\label{l2} Let $M$ y $N$ be smooth manifolds with volume forms $\eta$ y $\zeta$ respectively. Fix an smooth submersion $\rho:M\mapsto N$ and
let $\nu$ be a $(m-k)$-form on $M$ such that $\eta=\nu\wedge\rho^*(\zeta)$. 
\begin{enumerate}
\item[a)] If $Y$ is a vector field on $M$, then
$$
\div_\nu Y(x)=\div_\eta(Y)-\div_\zeta(D\rho (Y))\circ\rho.
$$
Moreover, if $Y$ is tangent to each $M_\lb$ then $\div_\eta Y(x)=\div_{\mu_{\rho(x)}}Y_{\rho(x)}(x)$, where $Y_\lb$ denotes the restriction of $Y$ to $M_\lb$.
\item[b)] Let $J$ be a non-vanishing smooth function on $M$ and let $\omega=J^{-1}\eta$. Then, for each vector field $Y$ on $M$, we have that 
$$
\div_\omega(Y)=J\div_{\eta}(J^{-1} Y)=\div_\eta(Y)-J^{-1}Y(J).
$$
Similarly, if $L$ is a smooth manifold endowed with a volume form $\omega$, $Y$ is a vector field on $M$, and $\Psi:M\to L$ is a  diffeomorphism, then
$$
\div_\eta(Y)=J_\Psi\div_\omega(D\Psi(J_\Psi^{-1}Y))= \div_\omega(D\Psi Y)\circ \Psi-J_\Psi^{-1}Y(J_\Psi),
$$
where $J_\Psi$ is the Jacobian of $\Psi$. 
\item[c)]  $\div([X,Y])=X(\div(Y))-Y(\div(X))$, for every $X,Y\in \vect(M)$.

\end{enumerate}
\end{Lemma}
\begin{proof}
Since $\eta=\nu\wedge\rho^*(\zeta)$, we have that
$$
\div_\eta(Y)\nu\wedge\rho^*(\zeta) =\mathcal{L}_Y(\nu)\wedge\rho^*(\zeta)+\nu\wedge\rho^*(\mathcal{L}_{D\rho(Y)}\zeta)=\mathcal{L}_Y(\nu)\wedge\rho^*(\zeta)+\left(\div_\zeta(D\rho (Y))\circ\rho\right)\nu\wedge\rho^*(\zeta).
$$
Since $\rho^*(\zeta)(x)$ vanishes on $TM_{\rho(x)}$, the latter identity implies that 
$$
\left(\div_\eta(Y)-\div_\zeta(D\rho (Y))\circ\rho)\right)\nu|_{M_\lb}=\mathcal{L}_Y(\nu)|_{M_\lb}=\div_\nu (Y)\nu|_{M_\lb}
$$
and this shows the first claim of part a). The second claim follows from the previous identity, and the coarea formula (see Appendix A or \cite{Fe}) implies that the restriction of $\nu$ to any $M_\lb$ coincides with $\mu_\lb:=J^{-1}\eta_\lb$.

The first claim in part b) follows by noting that $\mathcal L_X\omega=J^{-1}\mathcal L_X\eta +X(J^{-1})\eta$ and $J X(J^{-1})=-J^{-1}X(J)$. The second claim in part b) follows from the same argument, but using that $J_\Psi\eta=\Psi^*(\omega)$ and that the pullback of $\Psi$ exchanges the Lie derivatives.

To prove c), we compute
$$
\begin{aligned}
\div([X,Y])\eta &= \mathcal{L}_{[X,Y]}\eta=\mathcal{L}_X\mathcal{L}_Y\eta-\mathcal{L}_Y\mathcal{L}_X\eta= \mathcal{L}_X\div(Y)\eta-\mathcal{L}_Y\div(X)\eta \\
&= \big(X(\div(Y))+\div(Y)\div(X)-Y(\div(X))-\div(X)\div(Y)\big)\eta \\
&= \big(X(\div(Y))-Y(\div(X))\big) \eta.
\end{aligned}
$$

\end{proof}
\begin{Remark}
{\rm
It is not difficult to show the existence of $\nu$ such that $\eta=\nu\wedge\rho^*(\zeta)$, but clearly it is not unique. For example, we can define $\nu$ locally applying the implicit function theorem. For details, see Lemma 2.1 in \cite{Be}. 
}
\end{Remark}

Now we can prove the derivation formula given in Theorem \ref{smoothint}.

\begin{proof}[Proof of Theorem \ref{smoothint}]
We shall separate the proof in three cases: $N=\R$ and $X=\frac{\partial}{\partial\lb}$, $N=\R^k$ and $X$ is arbitrary, and finally the general case.

{\bf Case 1, $N=\R$ and $X=\frac{\partial}{\partial\lb}$.}

Fix $\lb^0\in\R$ and let $h>0$ be small enough. Consider the submanifold with boundary
$$
M_{\lb^0,\lb^0+h}=\{x\in M \mid \lb^0\leq\rho(x)\leq\lb^0+h\}.
$$ 
Clearly $\partial M_{\lb^0,\lb^0+h}=M_{\lb^0}\cup M_{\lb^0+h}$ and the outgoing normal vector is $v=\frac{1}{||\nabla\rho||}\nabla\rho$. Therefore, the divergence theorem implies that 
$$
\int_{M_{\lb^0+h}}f\mu_{\lb^0+h}-\int_{M_{\lb^0}}f\mu_{\lb^0}= \int_{M_{\lb^0,\lb^0+h}}\div(f ||\nabla\rho||^{-1} v)\eta.
$$
Using Lemma \ref{l1} and coarea formula (see Appendix A or \cite{Fe}) on the right hand side, we obtain
$$
F(\lb^0+h)-F(\lb^0)=\int_{\lb^0}^{\lb^0+h}\left(\int_{M_\lb}||\nabla\rho||^{-1}\div\left(f \frac{\check\partial}{\partial\lb}\right)\eta_\lb\right)\de\lb.
$$
Therefore, the fundamental theorem of calculus implies that 
 $$
\frac{\partial F}{\partial\lb}(\lb)=\int_{M_\lb}\div\left(f \frac{\check\partial}{\partial\lb}\right)\mu_{\lb}.
$$

{\bf Case 2, $N=\R^k$ and $X$ is arbitrary.}

Let $\rho=\rho_1\times\cdots\times\rho_k$. Notice that
$$
F(\lb)=\int_{M_\lb}(||\pi^i(\nabla\rho_i)||J^{-1}f)||\pi^i(\nabla\rho_i)||^{-1}\eta_\lb.
$$
Applying the previous case with $M=M^i_\lb$ defined in \eqref{mjlb} and using Lemma \ref{l1}, we obtain
 $$
 \begin{aligned}
 \frac{\partial F}{\partial\lb_i}(\lb)&=\int_{M_\lb}\div\left(||\pi^i(\nabla\rho_i)||J^{-1} f \frac{\check\partial}{\partial\lb_i}\right)||\pi^i(\nabla\rho_i)||^{-1}\eta_{\lb} \\
&=\int_{M_\lb}J_i\div\left(J_i^{-1} f \frac{\check\partial}{\partial\lb_i}\right)\mu_{\lb}.
\end{aligned}
$$
Since $\frac{\check \partial}{\partial\lb_i}\in TM^i_\lb$ and divergence in the previous identity is computed on $M^i_\lb$ with respect to $\eta^i_\lb$, Lemma \ref{l2} implies our result for $f\in C_c^r(M)$ and $X=\frac{\partial}{\partial\lb_i}$. If $X=\sum_i a_i\frac{\partial}{\partial\lb_i}$, with $a_i\in C^\infty(\R^k)$, then
$$
\begin{aligned}
\int_{M_\lb}\div(f \check X)\,\mu_{\lb} & =\sum_i a_i(\lb)\frac{\partial F}{\partial\lb_i} (\lb)+\sum_i \int_{M_\lb}f \frac{\check\partial}{\partial\lb_i}(a_i\circ\rho)\,\mu_{\lb} \\
&=X F (\lb)+(\sum \frac{\partial a}{\partial\lb_i})(\lb)F(\lb).
\end{aligned}
$$
{\bf Case 3, arbitrary N.}

For the general case, take $\lb\in V\subseteq N$ and $\Psi:V\to W\subseteq\R^k$
a local coordinate. Notice that
$$
F(\lb)=\int_{M_\lb}(J_{\Psi}\circ\rho) f J_{\rho}^{-1}(J_{\Psi}\circ\rho)^{-1}\eta_\lb,
$$
where $J_{\Psi}$ is the Jacobian of $\Psi$. Since $J_{\Psi\circ\rho}=(J_{\Psi}\circ\rho) J_{\rho}$, applying the previous case to $\Psi\circ\rho$, we obtain
$$
\begin{aligned}
XF(\lb)&=\int_{M_\lb}[\div((J_{\Psi}\circ\rho) f \check X)-\div D\Psi X(\Psi(\lb))J_{\Psi}(\lb)f]J_{\Psi}^{-1}(\lb)J_{\rho}^{-1}\eta_\lb
\\
&=\int_{M_\lb}\div(f \check X)+[J_{\Psi}^{-1} X(J_{\Psi})-\div(D\Psi X)\circ\Psi](\lb)f\mu_\lb.
\end{aligned}
$$
Part b) of Lemma \ref{l2} finishes the proof.
\end{proof}
\begin{Remark}
{\rm
The only step where we used that $f$ has compact support was when we applied the divergence theorem over the space $M_{\lb^0,\lb^0+h}$, but such identity would also hold if $M_{\lb^0,\lb^0+h}$ is compact. Therefore, if $\rho$ is proper our derivation formula holds for any $f\in C^r(M)$.
}
\end{Remark}
\begin{Remark}
{\rm
Clearly the map 
$$
\tilde F(\lb)=\int_{M_\lb}f\eta_{\lb}\,,
$$
obtained by replacing $\mu_\lb$ by $\eta_\lb$, is also smooth and we have that
$$
X\tilde F(\lb)=\int_{M_\lb}J^{-1}\div(J f\check X )-\div X(\lb)f\,\eta_{\lb}=\int_{M_\lb}\div(f\check X )-\div X(\lb)f+J^{-1}\check X(J)f\,\eta_{\lb}.
$$

} 
\end{Remark}



\subsection{Smooth structure and trivialization} \label{secsst}

In this subsection, we denote by $H\to N$ the field of Hilbert spaces $\mathcal{H}(\lb)$ defined in \eqref{hilbertexample}. We will use Theorem \ref{smoothint} to endow $H\to N$ with a smooth structure. However, it is not clear if this smooth field of Hilbert spaces admits a trivialization, unless further geometrical conditions are assumed. For instance, we apply Corollary \ref{corisoc} to show that if $\rho$ is also a proper map, then $p:H\to N$ admits a trivialization. Indeed, under that assumption $\rho$ defines a fiber bundle (Ehresmann's Theorem). In particular, there is a smooth manifold $K$ such that  locally each $M_\lb$ is diffeomorphic to $K$. Using a trivialization of the vector bundle $\pi: E\to M$ we can define a the required unitary operators $T(\lb)$ (see Theorem \ref{trivex1}).

Since $F\cong\pi^{-1}(x)$ is a finite-dimensional Hilbert space, it follows that $E\to M$ can be considered as a field of Hilbert spaces. Let $\Gamma_c^\infty(E)$ be the space of compact supported sections. 
Notice that for a given $\varphi\in\Gamma_c^\infty(E)$ the restriction 
$\varphi|_{M_\lambda}$ lies in $\mathcal{H}(\lb)$ for all $\lambda\in N$. Moreover, we can extend sections defined on $M_\lb$ to sections on $M$ by using the same argument applied to extend functions proposed in \cite[Lemma 5.34]{Le}.

\begin{Lemma}\label{compsec}
Let $\lambda\in N$ and let $f\in\Gamma^\infty_c(E_\lb)$ be a smooth section  with compact support. There exists a compact supported section $\varphi\in\Gamma^\infty_c(E)$\, such that $\varphi|_{M_\lambda}=f$.
\end{Lemma}

Since $\rho\equiv\lb$ on $M_\lb$, $\Gamma_c^{\infty}(E)$ is a $C^\infty (N)$-module, where the multiplication between $a\in C^\infty(N)$ and a section $\varphi$ is defined by $(a\circ\rho)\varphi$. 

Now we are in position to show Theorem \ref{hilejem}.

\begin{proof}[Proof of Theorem \ref{hilejem}]
The first two conditions of i)  in Definition \ref{def} are straightforward. To show the third one, we compute
$$
\begin{aligned}
\nabla_{X}(a \varphi) &=\nabla^{E}_{\check X}(a \varphi)+\frac{1}{2}(\div(a\check X)-\div(aX)\circ\rho) \varphi \\
&= \check X(a\circ \rho) \varphi+a\nabla^{E}_{\check X}( \varphi)+ \frac{1}{2}\big( \check X(a\circ\rho)+a\div(X) - X(a)- a\div(X) \big)\varphi 
\\
&= X(a) \varphi+a\nabla^{E}_{\check X}( \varphi)+ \frac{1}{2}\big(   X(a)+a\div(\check X) - X(a) -a\div(X) \big)\varphi
\\
&= X(a) \varphi+ a\nabla^{E}_{\check X}( \varphi)+a\frac{1}{2}(\div(\check X)-\div(X)\circ\rho) \varphi \\
&=X(a) \varphi+ a\nabla_X( \varphi)\,,
\end{aligned}
$$
where we have used the identity $\div(aX)=X(a)+a\div(X)$\,. The condition \emph{ii)} follows from Theorem \ref{smoothint}. Indeed, we compute for $X\in\vect(N)$,
$$
\begin{aligned}
Xh( \varphi, \psi)(\lambda) &= 
  \int_{M_\lb}\! \check X(h^E( \varphi, \psi))+\big(\div(\check X)-\div( X)\big)h^E( \varphi, \psi)\,\mu_{\lb} 
  \\&=
  \int_{M_\lb} \!h^E(\nabla^{E}_{\check X} \varphi, \psi)+ \frac{1}{2}\big(\div(\check X)-\div( X)\big)h^E( \varphi, \psi) \,\mu_\lambda 
\\&\;\qquad +\int_{M_\lambda}\!
h^E( \varphi,\nabla^{E}_{\!\bar{\check X}} \psi) +\frac{1}{2}\big(\div(\check X)-\div( X)\big)h^E( \varphi, \psi) \,\mu_\lambda 
\\&=
  \int_{M_\lb} \!h^E\big(\nabla^{E}_{\check X} \varphi+\frac{1}{2}\big(\div(\check X)-\div( X)\big) \varphi, \psi\big) \,\mu_\lambda 
\\&\;\qquad +\int_{M_\lambda}\!
h^E\big( \varphi,\nabla^{E}_{\check{\bar{X}}} \psi+\frac{1}{2}\big(\div(\check{\bar{X}})-\div(\bar{X})\big) \,\mu_\lambda \\
&= h(\nabla_X  \varphi, \psi)(\lambda)+h( \varphi,\nabla_{\bar{X}} \psi)(\lambda)\,,
\end{aligned}
$$
where $h^E$ denotes the Hermitian form in the field $E\to M$\,. 
Condition iii) follows by noticing that the set of compact supported smooth sections of $\Gamma^\infty(E_\lb)$ are dense in $\mathcal{H}(\lambda)$ and, by the Lemma \ref{compsec}  all compact supported section have a smooth compactly supported extension to $M$. 

We now compute the curvature of $\nabla$. 
$$
\begin{aligned}
\nabla_X\nabla_Y-\nabla_Y\nabla_X-\nabla_{[X,Y]} &=  R^E(\check{X},\check{Y})+\frac{1}{2}(\check X(\div\check Y)-\check Y(\div\check X)-\div[\check X,\check Y])\\
&\qquad +\frac{1}{2}( X(\div Y)- Y(\div X)-\div[ X, Y])\,.
\end{aligned}
$$
By part c) Lemma  \ref{l2}, it follows that $R(X,Y)= R^E(\check X,\check Y)$\,.
\end{proof}

We will show that, if $\rho$ defines a fiber bundle, then it induces a trivialization for field of Hilbert spaces $H\to N$ in a natural way. Let $K$ be the fiber of $\rho$. Thus, there is a open covering $\{U_i\}$ of $N$ and a family of difeomorphisms $\Phi_i:\rho^{-1}(U_i)\to U_i\times K$\, such that $\rho=\operatorname{proj}_1\circ\Phi_i$, where $\operatorname{proj}_1$ is the projection in the first coordinate. Let $(V_i,F,T^{E}_i)_{i\in I}$ be a local trivialization of the Hermitian finite-dimensional vector bundle $E\to M$. Refining the covering of $N$ if it is necessary, we can assume that 
$$
\rho^{-1}(U_i)= V_i.
$$
Recall that, for each $Y\in\vect(V_i)$, there is a smooth field of operators $\tilde{A}_i^{E}(Y):C^\infty(V_i,F)\to C^\infty(V_i,F)$ such that
\begin{equation}\label{trivet}
T^{E}_i\nabla^{E}_Yf= Y(T^{E}_i f )+ \tilde{A}_i^{E}(Y)T^{E}_if\,.
\end{equation}
For $\lambda\in N$, define the map 
\begin{equation}\label{deftl}
    \begin{aligned}
        &T_i(\lb): L^2(E_\lb)\to L^2(K,F) \\
        \big[T_i(\lambda) u\big](k) &= T^{E}_i(\Phi_i^{-1}(\lambda,k))\big[\big(J_{\Phi_i}^{-1/2}\cdot  u\big)\big( \Phi_i^{-1}(\lambda,k)\big)\big]\,,
    \end{aligned}
\end{equation}
where $J_{\Phi_i}$ denotes the Jacobian of $\Phi_i$. Since each $T^{E}_i(x)$ is unitary, the change of variable formula implies that each $T_i(\lb)$ is unitary. 
Let us prove that $T_i(\lb)$ defines a local smooth trivialization of the smooth field of Hilbert spaces $(H\to N,\Gamma^\infty_c(E),\nabla)$, where $\nabla$ is defined in Theorem \ref{hilejem}.

\begin{Theorem}\label{trivex1}
Let  $\rho:M\to N$ be a fiber bundle with fiber $K$ and let $\pi:E\to M$ be a finite-dimensional Hermitian vector bundle with fiber $F$. Let $\{U_i\}_{i\in I}$ be a covering of $N$ and $\Psi_i:\rho^{-1}(U_i)\to U\times K$  a local trivialization for $\rho:M\to N$, such that  $V_i=\rho^{-1}(U_i)$ is a open covering of $M$ over which we can define a local trivialization $(V_i,F,T^{E}_i)_{i\in I}$\,for $\pi:E\to M$.
Let $\tilde{A}_i^{E}$ and $T_i(\lb)$ defined by equation \eqref{trivet} and equation  \eqref{deftl} respectively.
Fix a volume form $\eta_0$ on $K$. If $\tilde{A}_i^{E}(\check X)\in C^\infty_{b}(\rho^{-1}(C),\mathcal B(F))$ for every $C\subset U_i$ compact and every $X\in\vect{(U_i)}$, then
the family $(U_i,L^2(K,F),T_i)_{i\in I}$ is a smooth local trivialization of the field $H\to N$ and the operators $\tilde{A}_i$ are given by
$$
(\tilde{A}_i(X)f)(\lb,k)=\tilde{A}_i^{E}(\check X)(\Phi_i^{-1}(\lb,k))f(\lb,k)\,,
$$
for each $f\in C^\infty(U_i, L^2(K, F))$.
\end{Theorem}

\begin{Remark}\label{prorho}
{\rm 
    Since $F$ is finite-dimensional, once one fix a basis, $\tilde{A}_i^{E}(\check X)\in  C^\infty_{b}(\rho^{-1}(C),\mathcal B(F))$ if and only if the entries of the corresponding matrix belongs to  $C^\infty_{b}(\rho^{-1}(C))$.  
    }
\end{Remark}

\begin{proof}
 
We will apply Corollary \ref{corisoc}. By definition, $\tilde{\Gamma}^\infty_i=T_i(\Gamma^\infty|_{U_i})=C^\infty_c(U_i\times K,F)$ which it is a dense smooth subspace of  $C^\infty(U_i,L^2(K,F))$.

It remains to show that the field of operators $\tilde{A}_i(X)$ and all its derivatives are locally uniformly bounded, for every $X\in\vect{(U_i)}$. Let us compute $\tilde{A}_i(X)$.
Let $\varphi\in\Gamma^\infty|_{U_i}=\Gamma_c^\infty(U_i,E)$, and we use the identification  $(T_i\varphi)(\lb,k)=[T_i(\lb)\varphi(\lb)](k)=[T_i(\lb)\varphi|_{M_\lb}](k)$\,. Using this, one can rewrite
$$
\big(T_i\varphi\big)= \big(T^{E}_i[(J^{-1/2}_{\Phi_i}\cdot \varphi)]\big)\circ \Phi_i^{-1} \,.
$$
Hence, considering $X$ as a field on $U_i\times K$ acting trivially on $K$, we have that 
$$
\big(XT_i\varphi\big)(\lb,k) = \big(\check{X}T^{E}_i[J_{\Phi_i}^{-1/2}\cdot \varphi]\big)(\Phi_i^{-1}(\lb,k))\,.
$$
Then, the relation \eqref{trivet} implies that 
$$
\begin{aligned}
\big(XT_i\varphi\big)(\lb,k) &= \big(\check{X}T^{E}_i[J_{\Phi_i}^{-1/2}\cdot \varphi]\big)(\Phi_i^{-1}(\lb,k)) \\
&=\Big(  T^{E}_i\nabla^{E}_{\check{X}}  \big[J_{\Phi_i}^{-1/2} \cdot \varphi\big]+ \tilde{A}_i^{E}(\check X)T^{E}_i\big[J_{\Phi_i}^{-1/2} \cdot \varphi\big]\Big)(\Phi_i^{-1}(\lb,k))\, .
\end{aligned}
$$
Using part b) of Lemma \ref{l2} we have that 
$$
\begin{aligned}
    \nabla^{E}_{\check{X}}  \big[J_{\Phi_i}^{-1/2} \cdot \varphi\big] &= \check X(J_{\Phi_i}^{-1/2})\varphi+ J_{\Phi_i}^{-1/2} \nabla^{E}_{\check X}\varphi \\
    &= J_{\Phi_i}^{-1/2}\Big(-\frac{1}{2}J_{\Phi_i}^{-1}\check X(J_{\Phi_i}) \varphi+  \nabla^{E}_{\check X}\varphi\Big) \\
    &=J_{\Phi_i}^{-1/2}\Big(-\frac{1}{2}(\div(\check X)-\div(X)\circ\rho)\varphi+  \nabla^{E}_{\check X}\varphi\Big) \\
    &=J_{\Phi_i}^{-1/2} \nabla_{\check{X}} \varphi\,.
\end{aligned}
$$
Hence, we obtain 
$$
(X T_i \varphi)(\lb,k) =\Big(T^{E}_i\big[J_{\Phi_i}^{-1/2} \nabla_{\check{X}} \varphi\big] 
+ \tilde{A}_i^{E}(\check{X})T^{E}_i [J_{\Phi_i}^{-1/2}\varphi]\Big) (\Phi_i^{-1}(\lb,k)) .
$$
Then, $XT_i\varphi= T_i\nabla_{\check X}\varphi+ (\tilde{A}_i^{E}(\check X)\circ\Phi^{-1}) T_i\varphi$. Therefore, we get that

$$
(\tilde{A}_i(X)f)(\lb,k)=\tilde{A}_i^{E}(\check X)(\Phi_i^{-1}(\lb,k))f(\lb,k)
$$
for $f\in C^\infty_c(U_i\times K, F)$\,. If $g\in C^\infty_c(K, F)$, the latter identity implies that 
$$
(\tilde{A}_i(X)g)(\lb,k)=\tilde{A}_i^{E}(\check X)(\Phi_i^{-1}(\lb,k))g(k)
$$
Thus, 
$$
[X_1\cdots X_n(\tilde{A}_i(X))g](\lb,k)=\check X_1\cdots \check X_n [\tilde{A}_i^{E}(\check X)](\Phi_i^{-1}(\lb,k))g(k),
$$
for every $g\in C^\infty_c(K, F)$ and $X_1,\cdots X_n,X\in\vect{(U_i)}$. Finally, 
$$
\sup_{\lb\in C}\|X_1\cdots X_n[\tilde{A}_i(X)](\lb)g\|\leq \sup_{x\in \rho^{-1}(C)}\left\{ \big\|\hat X_1\cdots \hat X_n [\tilde{A}_i^{E}(\check X)](x)\big\|\right\}\|g\|. 
$$
Since each $\tilde{A}_i^{E}(\check X)\in C^\infty_{b}(\rho^{-1}(C),\mathcal B(F))$, this finishes our proof.

\end{proof}

\begin{proof}[Proof of Proposition \ref{prosub}]
Since $\rho$ is proper, Ehresmann's Theorem implies that $\rho$ is a fiber bundle. Since $\rho^{-1}(C)$ is compact, $\tilde{A}_i^{E}(\check X)\in C^\infty(\rho^{-1}(C),\mathcal B(F))=C^\infty_b(\rho^{-1}(C),\mathcal B(F))$, for every $C\subset U_i$ compact and every $X\in\vect{(U_i)}$.     
\end{proof}

There are two particularly interesting cases that we would like to consider: the trivial line bundle $E=M\times F$, where $F$ is a finite-dimensional vector space, and the tangent bundle $E=TM$ endowed with the Levi-Civita connection.  

\begin{Corollary}\label{TlB}
Let $\rho:M\to N$ be a smooth submersion. Consider the field of Hilbert spaces $H\to N$ with $\Hi(\lb)=L^2(M_\lb;\mu_\lb)\otimes F=L^2(M_\lb,F)$.
\begin{enumerate}
    \item[i)] $H\to N$ together with the map $\nabla_X :C_c^{\infty}(M,F)\to C_c^\infty(M,F)$ given by
$$
\nabla_{X}(f)=\check X(f)+\frac{1}{2}(\div(\check X)-\div(X)\circ\rho)f
$$
becomes a flat smooth field of Hilbert spaces. 

    \item[ii)]  Assume that $\rho$ is a smooth fiber bundle and 
 $(U_i,K,\Phi_i)_{i\in I}$ be a smooth trivialization for $\rho$. Fix a volume form on $K$.  The family $(U_i,L^2(K,F),T_i)_{i\in I}$
    defines a smooth local trivialization of $H\to N$, where $T_i(\lb):L^2(M_\lb,F)\to L^2(K,F)$ is given by
    $$
    T_i(\lb)u(k)= J^{-1/2}_{\Phi_i} u(\Phi^{-1}_i(\lambda,k))\,.
    $$ 
Moreover, we have that $T_i\nabla_X\varphi=XT_i\varphi$, for every $X\in\vect(U_i)$ and $\varphi\in C^\infty_c(M,F)$.    
\end{enumerate}

\end{Corollary}

Recall that the tangent bundle $TM$ of a Riemannian manifold $M$ admits a canonical Hermitian structure. The corresponding connection is called the Levi-Civita connection and we will denoted by $\nabla^L$. By definition the sections are vectors fields over $M$. In a local trivialization $(V_i,F)$,  the Levi-Civita connection is determinate by the so called Christoffel symbols $\Gamma_{kj}^l$, which are defined by the identity 
$$
\nabla^L_{\partial_k}\partial_j= \sum_l \Gamma^l_{kj}\partial_l\,.
$$
Thus, in our notation $A_i(\partial_j)$ in the canonical basis induced by the coordinate    system is the matrix with entrances $\Gamma_{kj}^l$. Then, we have the following result.  

\begin{Corollary}\label{LeCi}
Let $\rho:M\to N$ be a smooth submersion. Consider the field of Hilbert spaces $H\to N$ with $\Hi(\lb)=L^2((TM)_\lb)$. Let $\vect_c(M)$ be the space of smooth vector fields with compact support  over $M$ and $\nabla^L$ denotes the Levi-Civita connection.
\begin{enumerate}
\item[i)] $H\to N$ together with the map $\nabla:\vect(N)\times \vect_c(M)\to\vect_c(M)$ given by
$$
\nabla_XY=\nabla^{L}_{\check X}Y + \frac{1}{2}\big( \div(\check X)-\div(X)\circ \rho)Y\,
$$
becomes a smooth field of Hilbert spaces with curvature $R(X,Y)=R^g(\check X,\check Y)$, where $R^g$ is the Riemannian curvature of $M$. Moreover, the following formula holds
\begin{equation}\label{torsion}
    \nabla_X\check Y-\nabla_Y\check X= [\check X, \check Y]+ \frac{1}{2}\big( \div(\check X)-\div(X)\circ \rho)Y - \frac{1}{2}\big( \div(\check Y)-\div(Y)\circ \rho)X.
\end{equation}

\item[ii)] Assume that $\rho$ is a smooth fiber bundle and  $(U_i,K,\Phi_i)_{i\in I}$ be a smooth trivialization for $\rho$. Fix a volume form on $K$. If the Christoffel symbols $\Gamma_{kj}^l$ of  the Levi-Civita connection belongs to $C_b^{\infty}(\rho^{-1}(C))$ for each compact set $C\subset U_i$ and $i\in I$, then the family $(U_i,L^2(K,F),T_i)_{i\in I}$
defines a smooth local trivialization of $H\to N$, where $T_i(\lb):L^2(M_\lb,\mu_\lb)\to L^2(K,F)$ is given by equation \eqref{deftl}.
\end{enumerate}

\end{Corollary}

\begin{proof}
    The only claim that does not follow from Theorems \ref{hilejem} and \ref{trivex1} is  equation \eqref{torsion}, but this is a direct consequence of the fact that the Levi-Civita connection is torsion free.  
\end{proof}

\subsection{Comparing with the holomorphic case}\label{Hcase}

In this subsection, we will analyze and discuss the similarities and differences between the construction of holomorphic direct images given in \cite{LS} and our Riemannian direct images. However, we should note that a proper comparison can only be made within a common framework. For instance, this will be the case if the involved manifolds are K\"{a}hler. We will leave this analysis for another opportunity.  

Let us briefly recall the framework in \cite{LS}. Let $M$ and $N$ be complex manifolds and let $\rho:M\to N $ be a holomorphic submersion. Fix a smooth form $\nu$ on $M$ such that it restricts to a volume form on each $M_\lb=\rho^{-1}(\{\lb\})$. Also, let $E=\pi:E \to M$ be a holomorphic vector bundle of finite rank with a Hermitian structure $h^E$. The required field of Hilbert spaces is defined just as we did after Definition \ref{xjmu}, but within the holomorphic category. Indeed, for each $\lb\in N$, let $E_\lb=\pi^{-1}(M_\lb)$, $E_\lb=\pi:E_\lb\to M_\lb$ and denote $\Hi_\lb$ the Hilbert spaces of holomorphic $L^2$-sections of $E_\lb$ with inner product given by
$$
h(\varphi,\psi)=\int_{M_\lb}h^{E}(\varphi,\psi)\nu.
$$
The construction of $\Gamma^\infty$ and $\nabla$ is described in detail in subsections 6.2 and 6.3 of Chapter II in \cite{LS}. We can notice at this point two main differences with our construction. First, $\Gamma^\infty$ is built up inductively in \cite{LS}, while in our case it is provided explicitly (it is the space compact supported sections). Secondly, in \cite{LS} the connection $\nabla$ is defined using the Chern connection on $E$ \cite{AMo}, while in our case we use any Hermitian connection. 

The main issue with the construction of the smooth structure on the field of Hilbert spaces $\{\Hi_\lb\}_{\lb\in N}$ is to guarantee that condition {\it iii)} in Definition \ref{def} holds. In order to overcome this problem, Lempert and Sz\"oke proved that if we required a certain geometrical property of $E$ together with the existence of a space of sections $A\subset C^\infty(E)$ satisfying certain (sharp) properties, then all the conditions in Definition \ref{def} hold. Let us explain in more detail their result.  Let us denote by $T^{1,0}M$ the holomorphic tangent bundle over $M$ and by $T^{0,1}M$ the antiholomorphic tangent bundle over $M$. We will denote by $\vect'{M}$ and $\vect''{M}$ the spaces of $(1,0)$ and $(0,1)$ vector fields respectively. Let us also denote by  $B_\lb:\Hi(\lb)\to H_\lb$ the Bergman projection. The required conditions are the following: 
\begin{enumerate}
\item[(i)] There is a family $\Xi\subset\vect'{N}$ that spans the bundle $T^{1,0}N$, and each $X\in\Xi$ has an integrally complete lift $\hat{X}\in\vect'{M}$.
\item[(ii)] There is a subspace $A\subset C^{\infty}(M,E)$ with the following properties. If $\varphi\in A$, then
\begin{enumerate}
             \item[a)] $\int_{M_\lb} h^E(\varphi)d\nu \in\R$ depends continuously on $\lb\in N$; and
             \item[b)] if $X\in \Xi$ and $\bar Y= X$, then $\div(\hat X)\varphi $, $\nabla^E_{\hat X}\varphi$,  $\nabla^E_{\hat Y}\varphi$, and $B\varphi\in A$. Further,
             \item[c)] if $u\in H_\lb$ and $\epsilon>0$, there is a $\varphi\in A$ such that $\int_{M_\lb} h^E(\varphi-u)d\nu<\epsilon$
         \end{enumerate}
     \end{enumerate}
     \begin{theorem*}{\cite[Theorem 7.2.1]{LS}}
        If conditions (i) and (ii) above hold then $H\to N$ together with $\Gamma^\infty$ and $\nabla$, defined in subsections 6.2 and 6.3 of Chapter II in \cite{LS}, becomes a smooth field of Hilbert spaces.
     \end{theorem*}
Let us explain how conditions (i) and (ii) are somewhat hidden in our construction. Regarding condition (i), note that in our framework the existence of a natural lift of vector fields on $N$ is guaranteed by the Riemannian structure and it is also crucial in the definition of our connection (see Definition \ref{xjmu} and Theorem \ref{hilejem}). However, we do not need to require that such a lift is integrally complete. This condition, together with (ii) a) and (ii) b), allows us to compute derivatives of functions of the form $F(\lb)=\int_{M_\lb}h^{E}(\varphi,\psi)d\nu$ (see Lemmas 7.3.2 and 7.3.3 in \cite{LS}). Instead, since we are working with compactly supported sections, we were able to apply the divergence theorem to show a similar result (see Theorem \ref{smoothint} and its proof). Finally, condition (ii) c) is meant to guarantee that condition (iii) in Definition \ref{def} is satisfied. Again, this is trivially satisfied by the space of compactly supported sections (see Lemma \ref{compsec}). To summarize, due to the inherent technical difficulties of working within the holomorphic category, Lempert and Sz\"oke found sharp conditions, meant to imply suitable properties, to ensure that certain (abstract) space sections and certain connection define a smooth field of Hilbert spaces. Meanwhile, in our case those properties were obtained from the Riemannian structure, and the required space of sections and the connection were defined explicitly.

\appendix

\section{Weakly smoothness}\label{appa}

The following result from \cite{LS} is applied many times in Section \ref{section2}. We decided to include it for a self-contained presentation. We also take this opportunity to fix a typo in the original article. Essentially, this result asserts that weakly smoothness implies norm smoothness.

\begin{Lemma} \cite[Lemma 5.1.1]{LS} 
Let $N$ be a finite-dimensional manifold and $\Hi$ be a Hilbert space with inner product $\langle\cdot,\cdot\rangle$. Also, let $f\in C^{n-1}(N;\Hi)$, $n\in\mathbb{N}$. If for every $X\in\vect(N)$  there is $f_X\in C^{n-1}(N;\Hi)$ such that 
$$
\langle f,\theta \rangle\in  C^{n}(N)\quad \text{ and }\quad X\langle f,\theta \rangle=\langle f_X,\theta \rangle\,,\quad \theta\in\Hi\,
$$
then $f\in C^{n}(N;\Hi)$ and $Xf=f_X$.
\end{Lemma}

\section{Coarea Formula}\label{appcoarea}

For the sake of a self-contained presentation of our results, we recall the so-called coarea formula in this appendix. 

Let $M$ and $N$ be oriented Riemannian manifolds of dimension $m$ and $k$ respectively, with $k<m$, and let $\rho:M\to N $ be a submersion. Also, let $\eta$ and $\zeta$ be the corresponding Riemannian volume forms, and
 $J_\rho$ as in Definition \ref{xjmu}.  Then for any measurable function  $f$ on $M$, the coarea formula asserts that 
 \begin{equation}\label{gcoarea}
\int_{M}f(x)J_\rho(x) \eta(x)=\int_{N}\left(\int_{M_\lb}f(z)   \eta_\lb(z)\right) \zeta(\lb),
\end{equation}
 whenever $f$ is nonnegative or $f J_\rho\in L^1(M)$. 

Since $J_{\rho}(x)=[\det(D\rho(x)D\rho(x)^*)]^{1/2}$, the function $J_\rho$ is well-defined even if $\rho$ is a smooth map. Moreover, the Morse-Sard Theorem implies that the coarea formula holds even if $\rho$ is not a submersion. Note that in such a case the set of regular points forms an open set in $M$.
\begin{Remark}
{\rm
For $M=\R^m$, coarea formula can be found in \cite{Fe} Theorem 3.2.12. It is stated using the $(n-k)$-Hausdorff measure restricted to $M_\lb$, but it is well known that it coincides with $\eta_\lb$ in our case.  The result for Riemannian manifolds follows from the case $M=\R^n$. 
}
\end{Remark}


Recall that, we introduce in Definition \ref{xjmu}  the volume form  $\mu_\lb=J^{-1}\eta_\lb$. Then, the coarea formula becomes the identity 
$$
\int_{M}f(x) \eta(x)=\int_{N}\left(\int_{M_\lb}f(z)   \mu_\lb(z)\right) \zeta(\lb).
$$
In particular, it follows that the map $T:L^2(M)\to\int_N^\oplus L^2(M_\lb,\mu_\lb)\zeta(\lb)$ given by $Tf(\lb)=f|_{M_\lb}$ is unitary.

\end{document}